\documentclass[a4paper,USenglish]{article}

\usepackage[left=2.5cm,right=2.5cm,top=2.5cm,bottom=3cm]{geometry}
\usepackage[pdfborder={0 0 0}]{hyperref}
\usepackage{enumitem}
\usepackage{amsmath}
\usepackage{amssymb}
\usepackage{amsthm}
\usepackage{mathtools}
\usepackage[table,xcdraw]{xcolor}
\usepackage{booktabs}
\usepackage{pdflscape}
\usepackage[leftcaption]{sidecap}
\sidecaptionvpos{figure}{t}
\sidecaptionvpos{table}{t}

\usepackage{tikz}
\tikzset{every picture/.style={baseline={(current bounding box.north)}}}

\newcommand\Alpha{\mathcal{A}}

\def\tgrid{\mathcal{T}}
\def\val{\mathrm{val}}

\newcommand{\Lang}[1]{\text{\normalfont\textsc{#1}}}
\newenvironment{parameterizedproblem}{%
  \leavevmode\nobreak\par
  \begin{list}%
    {}%
    {%
      \def\labelstyle{\itshape}
      \setlength{\topsep}{0pt}%
      \settowidth{\labelwidth}{\labelstyle Parameter:}%
      \setlength{\leftmargin}{\labelwidth}%
      \addtolength{\leftmargin}{\labelsep}%
      \setlength{\itemsep}{0pt}%
      \setlength{\parsep}{0pt}%
    }%
      \def\instance{\item[\labelstyle Instance:]}%
      \def\objective{\item[\labelstyle Objective:]}%
      \def\minimize{\item[\labelstyle Minimize:]}%
}{%
  \end{list}%
}

\newtheorem{problem}{Problem}
\newtheorem{example}{Example}
\newtheorem{lemma}{Lemma}

\newtheorem{theorem}{Theorem}
\newtheorem{observation}{Observation}
\newtheorem{reductionrule}{Rule}

\definecolor{jade}{rgb}{0.0, 0.66, 0.42}
\definecolor{cerise}{HTML}{CE4760}
\colorlet{fg}{jade!75!black}
\colorlet{bg}{cerise!75!black}

\usetikzlibrary{decorations.pathreplacing}
\usetikzlibrary{graphs}
\usetikzlibrary{arrows.meta}
\usetikzlibrary{shapes.geometric, arrows}
\usetikzlibrary{backgrounds}
\usetikzlibrary{3d}
\usetikzlibrary{calc, shapes.geometric}
\tikzset{
  dot/.style = {
    circle,
    inner sep     = 0pt,
    fill          = black,
    minimum width = 1.5mm
  },
  arc/.style = {
    semithick, ->, >={[round,sep]Stealth}
  }
}

\usepackage{listings}
\lstdefinelanguage{pseudocode}{
  morekeywords={
    algorithm,method,new,and,
    if,then,else,while,do,repeat,until,seq,
    seqdo,return,call,
    for,pardo,foreach,print,output,input,exit,
    break,loop,end,begin,goto,par,global,local,
    read,write,stop,idle,procedure,function,
    throw,catch,continue
  },
  sensitive=true,
  morecomment=[l]{//},
  morestring=[b]",
  morestring=[s]{``}{''},
}
\lstdefinestyle{pseudocode}{
  language=pseudocode,
  basicstyle=\small\rmfamily,
  commentstyle=\upshape\color{black!50},
  keywordstyle=\bfseries\itshape\color{fg},
  identifierstyle=\itshape,
  stringstyle=\rmfamily,
  columns=fullflexible,
  mathescape,
  literate={<-}{{$\gets$\ }}2,
  numbers=left,
  numberstyle=\scriptsize\sffamily,
}

\newcommand\Algo[1]{\text{\textcolor{bg}{\bfseries\itshape#1}}}
\newcommand\commented[1]{\textcolor{black!50}{#1}}

\title{The Keplerian Traveling Salesperson Problem}
\author{Max Bannach \and Giacomo Acciarini \and Dario Izzo}
\date{%
  European Space Agency, Noordwijk, 2201 AZ, The Netherlands\\
  \texttt{\{max.bannach, giacomo.acciarini, dario.izzo\}@esa.int}\\[2ex]
}

\begin{document}

\maketitle

\begin{abstract}
We address a fundamental challenge in space mission design and space
logistics: planning interplanetary trajectories for missions that must
rendezvous with multiple bodies. Such mission occur, for instance, in active debris removal, in-orbit
servicing, or asteroid belt exploration. We model these problems as a
variant of the Traveling salesperson problem (\Lang{tsp}),
which we term the Keplerian \Lang{tsp} (\Lang{ktsp}). Unlike the well-studied \Lang{tsp},
the \Lang{ktsp} accounts for the motion of orbital targets, leading to
time-dependent and asymmetric transfer costs that capture key
real-world effects in astrodynamics. 

We provide a rigorous formalization of the \Lang{ktsp} and release a
benchmark suite to support its study. Central to our approach is a
time-unfolding technique that reformulates the continuous problem as a
discrete optimization task in a time-expanded network. This
representation makes the benchmark accessible to researchers in
discrete optimization even without prior knowledge of celestial
mechanics. We also develop an alternative encoding as an integer
linear program using Interval-based Dynamic Discretization Discovery
to handle the time-dependent nature of transfers. We leverage
state-of-the-art \Lang{ilp} solvers to solve the \Lang{ktsp}
instances, accompanied by a detailed computational study that
highlights their strengths and limitations. We complement these exact methods with an
initial solution heuristic, an improvement heuristic, and preprocessing
routines that preserve optimality.
\end{abstract}

%
%
\section{Introduction}

The design of multi-rendezvous trajectories is a common topic in
advanced space mission planning, with applications spanning a wide
range of fields. Missions for on-orbit servicing
\cite{daneshjou2017mission}, active debris removal
missions~\cite{zhang2024global, izzo2018kessler, cerf2013multiple},
asteroid mining~\cite{zhang2025sustainable,yarndleymultiple}, and
exploration missions targeting the main belt (e.g., NASA’s DAWN
spacecraft \cite{rayman2006dawn, li2021sequence, shen2023dyson,
  zhang2025sustainable}), all involve solving variations of this
complex problem. These diverse mission types not only require a deep
understanding of orbital mechanics but also of advanced optimization
techniques capable of handling the complex, real-world constraints
underling the problem.  Multi-rendezvous trajectory design poses
complex combinatorial challenges that have motivated benchmark
problems through the Global Trajectory Optimization Competition
\cite{casalino2014problem, izzo2018kessler, grigoriev2014gtoc5,
  zhang2025sustainable, shen2023dyson} and sparked the development of
increasingly sophisticated methods \cite{li2021sequence,
  zhang2025hierarchical, huang2025rapid,
  bang2019multitarget}. Most existing approaches tackle the problem
primarily from an astrodynamics perspective, developing tailored
algorithms and leveraging domain-specific insights. These methods typically rely on heuristics and
hyperparameters to prune the search space, often sacrificing
guarantees of global optimality and limiting connections to
general-purpose combinatorial optimization.  In this paper, we address
these limitations by tackling the problem at its foundations: We
provide a rigorous mathematical formalization, establish provable
guarantees for a proposed algorithmic framework, and clarify the
inherent trade-offs of heuristic approaches. Recognizing that the
problem's complexity depends strongly on the precise mission
objectives, we introduce and focus on a deliberately simple yet
representative formulation that captures the essential combinatorial
and dynamical aspects of multi-rendezvous mission design.

We study multi-rendezvous trajectories as graph exploration
problems. The poster child of such exploration problems is the
\emph{traveling salesperson problem} (\Lang{tsp}), which asks to find
the minimum-cost round tour in a given graph that visits every vertex
\emph{exactly once.} In our setting, the vertices of the graph are
celestial bodies orbiting a common central mass following Kepler's
laws of planetary motion. To be more specific, the setup of our
problem is a central body $M$ and a set of celestial bodies~$\Alpha$
defined by their six Keplerian elements. The objective is to find a
trajectory that starts and ends at some $\alpha_s\in\Alpha$, that
visits (i.e., matches position and velocity) with every element in
$\Alpha$ at least once, and that minimizes the cumulative velocity
change $\Delta V$.  In order to have a well-defined problem, we also
require an initial epoch $t_0$ at which the missions starts and a
final epoch $t_{\max}$ at which the trajectory should reach $\alpha_s$
the latest.  Let us collect this definition of the \emph{Keplerian
traveling salesperson problem} (\Lang{ktsp}) formally:

\begin{problem}[\Lang{ktsp}]
\begin{parameterizedproblem}
\label{problem:ktsp}
  \instance A central massive body $M$, a time window
  $(t_0,t_{\max})$, a set of Keplerian objects $\Alpha$, an
  $\alpha_s\in \Alpha$ (i.e. the starting location) and an oracle
  returning the velocity change ($\Delta V$) necessary to connect two
  $\alpha, \alpha'\in \Alpha$ at epochs $t,t' \in
  [t_0,t_{\max}]$.  \objective Find a trajectory starting at $t_0$ at
  $\alpha_s$ that visits every element in $\Alpha$ and returns to
  $\alpha_s$ no later than $t_{\max}$.  \minimize The cumulative
  $\Delta V$ (instantaneous velocity changes) accumulated along the
  trajectory.\footnote{We assume that the transfer cost from $(\alpha,t)$
  to $(\alpha',t' )$ is a function solely of $t$ and $t'$ that does not
  depend on the history of the trajectory, e.g., a Lambert arc
  \cite{izzo2016designing}. It is not relevant for the problem
  definition which transfer costs are used exactly, and we may assume
  that these costs are provided via oracle access.}
\end{parameterizedproblem}
\end{problem}

\noindent
To situate the Keplerian traveling salesperson problem within the
broader landscape of \Lang{tsp} generalizations~\cite{SallerKK23},
consider the progression shown in Figure~\ref{figure:tsp}. The
\emph{moving target} variant (\Lang{mt-tsp}) (formally introduced in
2003~\cite{helvig2003moving}) extends the classical Euclidean case by
allowing each ``city'' to move with a constant velocity. The Keplerian
variant goes a step further by modeling targets as bodies evolving
under a central gravitational field, and by introducing a generic
\emph{travel cost} that depends on the departure \emph{and} arrival
epoch. Another related formulation is the \emph{time-dependent}
version (\Lang{td-tsp}) that was introduced in 1980~\cite{fox1980n},
where arc costs vary with the \emph{departure} time. In contrast, in
\Lang{ktsp} the cost ($\Delta V$) depends explicitly on both the
starting \emph{and} arrival times, reflecting the two-point boundary
value nature of orbital transfers and introducing the possibility of
waiting times at each body. This apparently small change has two major
consequences: First, there is no analogue of the FIFO
(first-in-first-out) property that underlies most efficient algorithms
for \Lang{td-tsp}, making much of the existing literature on
heuristics and solution methods inapplicable; second, the dependence
on both endpoints of the transfer prevents a straightforward
generalization of time-dependent cost functions (see, e.g.,
\cite{vu2023solving}).

\begin{figure}[htb]
  \input{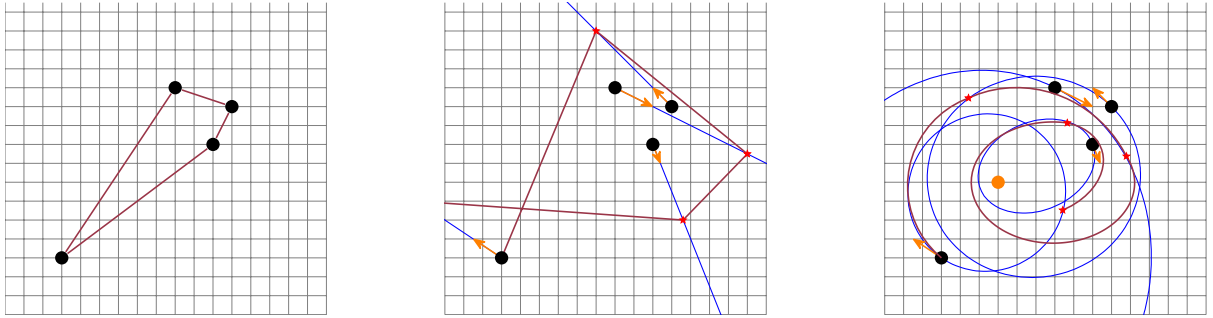}
  \caption{Illustration of the progression from classical Euclidean
    traveling salesperson problem (left), where cities are static points, to the
    \Lang{mt-tsp} variant (center), in which cities move with initial
    velocities, and finally to \Lang{ktsp} (right),
    where motion is governed by gravitational dynamics around a
    central mass. Velocity vectors and the central mass are shown in
    \textcolor{orange}{orange}, the paths along which the cities move
    are shown in \textcolor{blue}{blue}, a solution to each instance is
  shown in \textcolor{bg}{dark red}, whereby for variants with moving
  cities the interception points are highlighted with a
  \textcolor{red}{red} star.}
  \label{figure:tsp}
\end{figure}

\subsection{Our Contribution}

This article provides the first systematic study of the Keplerian
traveling salesperson problem. The contribution is
threefold:
\begin{description}
\item[Static and Dynamic Discretization]
We introduce and compare two \Lang{ilp} formulations for the
problem. The first relies on a static discretization using a
time-expanded network, yielding a straightforward but less efficient
encoding. The second employs interval-based dynamic discretization
discovery, allowing the model to refine the temporal discretization
adaptively and resulting in a more compact and effective formulation. 
\item[Heuristic Improvements Based on Domain Knowledge]
We enhance the performance of our \Lang{ilp} formulations using four
heuristic adjustments that accelerate runtime while preserving
optimality guarantees. First, we introduce a set of reduction rules
that simplify the problem during preprocessing. We then develop a
constructive heuristic for generating high‑quality initial solutions
to seed the \Lang{ilp} solver, as well as an improvement heuristic
that refines intermediate solutions explored during
branch-and-bound. Finally, we provide mechanisms to pre-compute a
subset of dynamically discovered constraints whenever it is certain
that they will inevitably be generated by the solver. 
\item[A Benchmark Set] We construct and openly release a benchmark set
  of diverse \Lang{ktsp} instances. The set is designed to cover a
  broad variety of problem scenarios, including missions in the
  asteroid belt and within the Jovian system. In this work, we use it
  to evaluate our approaches, and in future research it will enable
  the community to develop, compare, and benchmark algorithmic
  strategies for designing multi-rendezvous missions. 
\end{description}

\subsection{Related Work}
The travelling salesperson problem has been mentioned in previous
works related to spacecraft trajectory design. Early studies on
on-orbit servicing~\cite{bourjolly2006orbit} and space debris
removal~\cite{cerf2013multiple,izzo2015evolving} advertised the
advantages of thinking about specific spacecraft trajectory design
problems as \Lang{tsp} variants.
Since then, various algorithmic methods, primarily
heuristic-based, have been deployed to solve instances of these
problems~\cite{zhang2022timeline,hou2024traveling, federici2019time,
  lopez2022asteroid}. More recently, the use of integer linear
programming~(\Lang{ilp}) became a crucial component of the winning
strategy developed by NASA’s Jet Propulsion Laboratory (JPL) during
GTOC12 for the ``Sustainable Asteroid Mining'' challenge. A derived
methodology, inspired by JPL's success, was recently formalized and
tested in the same context~\cite{yarndleymultiple}.  Our
metaheuristics are original with this paper, but previous work
carried out in the context of GTOCs has explored similar ideas of
combining beam searches with local improvements \cite{zhang2023gtoc,
  izzo2016designing, simoes2017multi}.  The dynamic discretization of
time on which our encoding is based was recently discovered by Boland
et al.~\cite{BolandHMS17}. Following this work, it was applied in
various contexts, e.g., for train rescheduling~\cite{CroellaLMV24},
continuous-time service networks~\cite{MarshallBSH21}, supply chain
optimization~\cite{DykK24}, \Lang{tsp} with time
windows~\cite{VuHBS20}, multi-attribute two-echelon location
routing~\cite{EscobarVargasC24}, routing problem with out-and-back
routes~\cite{LagosBS22}, and time-dependent shortest path
problems~\cite{HeBNS22}.

\subsection{Structure of this Article}

In Section~\ref{sec:unroll}, we model \Lang{ktsp} using a
\emph{time-index formulation}, where continuous time is
\emph{unrolled} into a discrete grid. The following section  introduces the \emph{Dynamic
Discretization Discovery} (DDD) technique, which reduces the overall
problem complexity while maintaining mathematical guarantees on the
solution quality. To improve the scalability of our methods, we
incorporate preprocessing rules and heuristic improvements in Section~\ref{section:improvements}.
Finally, we present
a comprehensive benchmark set of \Lang{ktsp} instances in
Section~\ref{sec:benchmarkset} and evaluate our approaches on this
dataset in Section~\ref{sec:experiments}.

%
%
\section{Static Discretization via Time-Expanded Networks}
\label{sec:unroll}
There are two common methods to encode combinatorial problems with
temporal components into general-purpose methodologies:
\emph{continuous formulations} and \emph{time-indexed
formulations}~\cite{DykK24}. 
The former uses continuous variables to model the temporal
properties of the problem, while the latter discretizes
the problem using \emph{time points.} While continuous
formulations often perform poorly, time-indexed formulations struggle
with finding good trade-offs between their size (number of time
points) and their approximation
quality~\cite{CroellaLMV24}. Time-index formulations have, however,
various advantages. First, the time points are explicitly pre-computed,
which allows the pre-computation of all costs. Second, the formulation
only requires binary variables, which allows to apply various
constraint optimization techniques out of the box. Finally, the
recently introduced \emph{dynamic discretization discovery} techniques
enable it to scale more effectively than the continuous
formulation~\cite{BolandHMS17}. We will utilize dynamic discretization
discovery in
Section~\ref{section:ddd}.

\subsection{Unrolling Time with a Time-Expanded Network}
\label{section:time-expanded-networks}

The basic idea of \emph{time-index formulations} is to \emph{unroll}
the time into discrete points. In our context, we define a directed
graph $G$, called the \emph{time-expanded network}, to a \Lang{ktsp} instance $(M,\Alpha,t_0,t_{\max},\alpha_s)$
by picking a discrete, uniform, \emph{time grid} $\tgrid$ containing $t_0$ and
$t_{\max}$, with a separation $dt = (t_{\max} - t_0) / (|\tgrid| - 1)$
between points. The graph contains a vertex for every 
$\alpha\in\Alpha$ and every $t\in \tgrid$:
\[
  V(G)\coloneq \Alpha\times\tgrid.
\]
The edge set of $G$ consists of two parts: the \emph{coasting arcs} $E_C$
and the \emph{transfer arcs} $E_T$. The former just indicates that it is
possible to ``stay'' at a celestial body:
\[
  E_C\coloneq \big\{\,
  \big((\alpha,t),(\alpha,t+dt)\big)
  \mid
  \text{$\alpha\in\Alpha$, $\{t,(t+dt)\}\subseteq \tgrid$}
  \,\big\}.
\]
The transfer arcs describe feasible transfers: 
\begin{align*}
  E_T&\coloneq
  \big\{\,
  \big((\alpha,t),(\beta,t')\big)
  \mid
  \text{$\{\alpha,\beta\}\subseteq\Alpha$, $\{t,t'\}\subseteq \tgrid$, $t<t'$,}\\
  &\qquad\qquad\text{and there is a feasible transfer from $\alpha$ at
    epoch $t$ to $\beta$ at epoch $t'$}
  \,\big\}.  
\end{align*}
The edge set of $G$ now simply is $E(G)\coloneq E_C\cup E_T$.  We also
define \emph{weights} on the edges $w_G\colon
E(G)\rightarrow\mathbb{R}$ that describe the required $\Delta V$ to
perform the transfer. Hence, we have $w_G(e)=0$ for all $e\in E_C$ and
$w_G(e)\geq 0$ for $e\in E_T$.  Since we are
focusing on the \emph{global} trajectory optimization problem, we
assume oracle access to the required \(\Delta V\) for each
transfer. That is, we assume we have access to a function that
computes $w_G(e)$ solely from the pair \(\big((\alpha,t),(\beta,t')\big)\), independent
of the history of prior rendezvous. Figure~\ref{figure:time-expanded-network} illustrates the
introduced concepts and the resulting graph $G$.

\begin{SCfigure}[1.0][htbp]
  \centering
  \begin{tikzpicture}[
      label/.style = { midway, circle, fill=white, inner sep=0pt,
        font=\scriptsize\color{orange}}
    ]

    \foreach \a/\al in {1/3,2/2,3/1,4/0}{
      \node[anchor=east, baseline, color=bg] at (0.5,\a) {$\alpha_{\al}$};
      \foreach \t in {1,2,...,6}{
        \node[dot] (\a-\t) at (\t,\a) {};
      }      
    }
    \foreach [count=\c from 0] \t in {1,2,...,6}{
      \ifthenelse{\c<5}{
        \node[baseline, fg] at (\t,4.75) {$t_\c$};
      }{
        \node[baseline, fg] at (\t,4.75) {$t_{\max}$};
      }
    }

    \draw[bg, semithick, decorate,decoration={brace}]
    (-.25,.75) -- (-.25,4.25) node[midway,left] {$\Alpha$\,\,};
    \draw[fg, semithick, ->, >={[round]Stealth}]
    (1-0.25,5.25) -- node[midway, above] {Time} (6+0.25,5.25);

    \foreach \a in {1,2,3,4}{
      \foreach \t in {1,2,...,5}{
        \pgfmathtruncatemacro{\nextt}{\t+1}
        \draw[semithick, ->, >={[round,sep]Stealth}, color=gray] (\a-\t) -- (\a-\nextt);
      }
    }

    \draw[semithick, ->, >={[round,sep]Stealth}] (4-1) -- node[label] {$1$} (3-2);
    \draw[semithick, ->, >={[round,sep]Stealth}] (4-2) -- node[label] {$2$} (3-3);
    \draw[semithick, ->, >={[round,sep]Stealth}] (4-4) -- node[label] {$2$} (3-6);
    \draw[semithick, ->, >={[round,sep]Stealth}] (4-3) -- node[label] {$2$} (2-4);

    \draw[semithick, ->, >={[round,sep]Stealth}] (1-1) -- node[label] {$3$} (3-2);
    \draw[semithick, ->, >={[round,sep]Stealth}] (1-2) -- node[label] {$4$} (3-3);
    \draw[semithick, ->, >={[round,sep]Stealth}] (1-4) -- node[label] {$5$} (2-5);
    \draw[semithick, ->, >={[round,sep]Stealth}] (1-3) -- node[label] {$6$} (2-4);
    \draw[semithick, ->, >={[round,sep]Stealth}] (2-5) -- node[label] {$7$} (4-6);
    \draw[semithick, ->, >={[round,sep]Stealth}] (2-5) -- node[label] {$8$} (1-6);
    \draw[semithick, ->, >={[round,sep]Stealth}] (2-2) -- node[label] {$10$} (3-4);
    \draw[semithick, ->, >={[round,sep]Stealth}] (2-2) -- node[label] {$11$} (1-3);
    \draw[semithick, ->, >={[round,sep]Stealth}] (3-3) -- node[label] {$12$} (1-4);
    \draw[semithick, ->, >={[round,sep]Stealth}] (3-3) -- node[label] {$13$} (2-5);
  \end{tikzpicture}
  \caption{Illustration of a time-expanded network for a hypothetical
    \Lang{ktsp} instance. Every dot is a vertex of the graph, where
    rows correspond to celestial bodies and columns to
    time points. Coasting arcs are illustrated as gray arrows (they have cost zero), while feasible
    transfers are shown in black (they have a cost shown as orange number on top).}
  \label{figure:time-expanded-network}
\end{SCfigure}
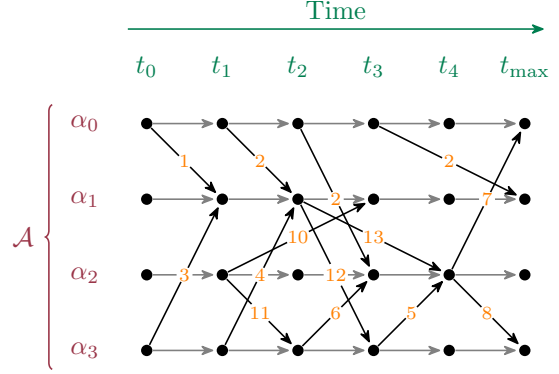
A path in a time-expanded network is called a \emph{trajectory.}
Observe how a trajectory that starts at
$(\alpha_s,t_0)$, ends at $(\alpha_s,t_{\max})$, and that visits
every object at least once corresponds a solution for
\Lang{ktsp}. We call such trajectories \emph{feasible}.
For a trajectory $P$, we let $\val(P)$ be the sum of the
weights of the arcs appearing in $P$, and for a time-expanded network
$G$ we let $\val(G)$ be the minimum over the value of all feasible
trajectories in $G$.
For instance,
a solution $P$ for $\alpha_s=\alpha_0$ with $\val(P)=25$ in
Figure~\ref{figure:time-expanded-network}~is:
\begin{align*}
(\alpha_0,t_0)
\stackrel{\textcolor{orange}{1}}{\rightarrow}
(\alpha_1,t_1)
\stackrel{\textcolor{orange}{0}}{\rightarrow}
(\alpha_1,t_2)
\stackrel{\textcolor{orange}{12}}{\rightarrow}
(\alpha_3,t_3)
\stackrel{\textcolor{orange}{5}}{\rightarrow}
(\alpha_2,t_4)
\stackrel{\textcolor{orange}{7}}{\rightarrow}
(\alpha_0,t_{\max}).
\end{align*}
\begin{observation}\label{observation:time-expanded-opt}
  For $dt\rightarrow 0$, the minimum $\Delta V$
  trajectory in $G$ that is feasible corresponds to an \emph{optimal} solution
  of the corresponding (continuous) \Lang{ktsp} instance. 
\end{observation}

%
%
\subsection{An Time-Indexed Integer Linear Programming Formulation for
KTSP}\label{section:csp}

It is rather straightforward to develop a time-indexed formulation for
\Lang{ktsp} using the time-expanded network $G$ and the fact that a
solution to \Lang{ktsp} corresponds to a route starting at
$(\alpha_s,t_0)$, ending at $(\alpha_s,t_{\max})$, and visiting for
every $\alpha\in\Alpha$ at least one $(\alpha,t)$ for some $t\in
\tgrid$ (Observation~\ref{observation:time-expanded-opt}). We
introduce a binary variable $x^{\alpha,\beta}_{t,t'}$ for every arc
$\big((\alpha,t),(\beta,t')\big)\in E(G)$ . The
semantics are that setting $x^{\alpha,\beta}_{t,t'}$ to true means
that the spacecraft takes the transfer from $\alpha$ to $\beta$
starting at epoch~$t$ and arriving at~$t'$. Since we want to minimize
the cumulative $\Delta V$ of the selected transfers, we have the
following objective function:
\begin{description}
\item[Objective Function] The cumulative $\Delta V$ of the selected
  transfers should be as small as possible:
  \[
  \mathrm{minimize}\quad  \sum_{\mathclap{((\alpha,t),(\beta,t'))\in E(G)}}\,\,
  x^{\alpha,\beta}_{t,t'} \cdot w_G\big((\alpha,t),(\beta,t')\big).
  \]
\end{description}
The following constraints ensure that the selected transfers ensemble
a feasible trajectory:
\begin{description}
\item[Departure Constraint] The tour \emph{departs} from every
  body at least once, i.e., for every $\alpha\in\Alpha$ there is
  at least one $\beta\in\Alpha\setminus\{\alpha\}$ and two epochs $t,t'\in \tgrid$ with  $x^{\alpha,\beta}_{t,t'}=1$. Formally: for all $\alpha\in\Alpha$:
  \[
  \sum_{\mathclap{((\alpha,t),(\beta,t'))\in E_T(G)}}\,\,x^{\alpha,\beta}_{t,t'}
  \quad
  \geq 1.
  \]
\item[Flow Constraint] Whenever the tour reaches or leaves
  an $\alpha\in\Alpha$ at some epoch $t$, it must
  leave or reach $\alpha$ at epoch $t$, i.e.,
  for all
  $\alpha\in\Alpha$ and all $t\in \tgrid$ with $(\alpha,t)\not\in\{(\alpha_s,t_0),(\alpha_s,t_{\max})\}$:
  \[
  \sum_{\mathclap{((\alpha,t),(\beta,t'))\in E(G)}}\,\,x^{\alpha,\beta}_{t,t'}
  \quad
  -
  \qquad
  \sum_{\mathclap{((\gamma,t''),(\alpha,t))\in E(G)}}\,\,x^{\gamma,\alpha}_{t'',t}
  \quad
  =0.
  \]  
\item[Vehicle Constraint] There is only one spacecraft, i.e., for the
  initial body $\alpha_s$ and epoch $t_0$ we have:
  \[
  \sum_{\mathclap{((\alpha_s,t_0),(\beta,t))\in E(G)}}\,\,x^{\alpha_s,\beta}_{t_0,t}
  \leq
  1.
  \]  
\end{description}

Let us denote with $\Gamma_{\Lang{ktsp}}$ the \emph{integer linear
program} (\Lang{ilp}) that consists of the just mentioned objective
function and constraints. The following lemma observes that an optimal
solution to $\Gamma_{\Lang{ktsp}}$ corresponds to an optimal solution
to the corresponding \Lang{ktsp} instance.

\begin{lemma}
\label{lemma:equivalence}
For $dt\rightarrow 0$, an optimal solution to $\Gamma_{\Lang{ktsp}}$ corresponds to an
optimal solution to the corresponding \Lang{ktsp}.
\end{lemma}
\begin{proof}
By the \textbf{Flow Constraint}, a route associated with a
solution to $\Gamma_{\Lang{ktsp}}$ cannot start at an
$\alpha\in\Alpha\setminus\{\alpha_s\}$ or some $t\in \tgrid\setminus\{t_0\}$. By
the \textbf{Departure Constraint}, there must be a route and, hence,
there must be a route starting at $(\alpha_s,t_0)$. Due to the
\textbf{Flow Constraint}, this trajectory ends in
$(\alpha_s,t_{\max})$. Since the time-expanded network is acyclic, the
\textbf{Departure Constraint} in conjunction with the \textbf{Vehicle
  Constraint} implies that there is a single trajectory
visiting all elements of $\Alpha$. By
Observation~\ref{observation:time-expanded-opt}, this trajectory
corresponds for $dt\rightarrow 0$ to an optimal solution for \Lang{ktsp}.
\end{proof}

\section{Dynamic Discretization via Time-Interval Networks}
\label{section:ddd}

Time-indexed encodings come with an undesirable trade-off: fine temporal grids
(small $dt$) provide solution accuracy approaching the continuous
optimum but generate intractably large problems, while coarse grids
remain computationally manageable but introduce significant
discretization errors that can render solutions meaningless for
mission design. Even worse, time-indexed encodings scale poorly with the
granularity of the time grid. Their size scales as
$O(|\Alpha|\cdot|\tgrid|)$, making it difficult even to write down the
encoding explicitly. (Note that $|\tgrid|$ may be exponential in the
input if $t_0$ and $t_{\max}$ are given in binary.)
\begin{example}
  An instance with $|\Alpha| = 20$
  objects over a 5-year mission window discretized at daily
  resolution requires $|\mathcal T| = 1825$ time points, resulting in
  a time-expanded network with $|V(G)|$ = 36500 vertices and
  millions of edges in $E(G)$. The resulting time-indexed
  formulation becomes computationally prohibitive, with memory
  requirements alone exceeding practical limits before optimization
  even begins.
\end{example}

\noindent
This quality-size trade-off can be circumvented by a strategy introduced
by Boland at al.~\cite{BolandHMS17}: \emph{Dynamic Discretization
Discovery} (DDD), which adaptively refines the temporal
discretization only where needed, i.e., when temporal constraints are
violated, rather than uniformly across the entire time horizon.

\subsection{Unrolling Time Dynamically with a Time-Interval Networks}

We follow the approach of Marshall et al.~\cite{MarshallBSH21} and implement DDD using
intervals. The idea is as follows: Instead of adding a node
$(\alpha,t)$ for every $\alpha\in\Alpha$ and every
epoch $t\in\tgrid$ to the time-expanded network, we define a \emph{coarse}
set of \emph{intervals} $\Lambda(\alpha)$ for every
$\alpha\in\Alpha$ that \emph{partitions}~$\tgrid$, i.e., for every $t\in \tgrid$
there is exactly one $\lambda\in\Lambda(\alpha)$ that contains
$t$. Recall from Section~\ref{section:time-expanded-networks} that in
a time-expanded network we connected two nodes $(\alpha,t)$ and
$(\beta,t')$ with a directed edge if there is feasible transfer from $\alpha$
at epoch $t$ reaching $\beta$ at epoch $t'$. The weight of this edge is defined as the minimum possible $\Delta V$ needed for this transfer.
In a \emph{time-interval network}, we connect a node $(\alpha,\lambda_{\alpha})$ to
$(\beta,\lambda_{\beta})$ with $\lambda_{\alpha}\in\Lambda(\alpha)$ and
$\lambda_{\beta}\in\Lambda(\beta)$ if there is a $t\in \lambda_{\alpha}$
and a $t'\in\lambda_{\beta}$ such that there is a feasible transfer from $\alpha$ at epoch $t$ that reaches $\beta$ at epoch $t’$. The weight of
this edge is the \emph{minimum} $\Delta V$ over all possible choices
of $t$ and $t'$, see Figure~\ref{figure:time-intervalled-network}.

\begin{SCfigure}[2.0][htbp]
  \centering
  \begin{tikzpicture}   
    \draw[semithick, bg, |-|] (1,4) -- (2,4);
    \draw[semithick, bg, |-|] (3,4) -- (4,4);
    \draw[semithick, bg, |-|] (5,4) -- (6,4);
    \draw[semithick, bg, |-|] (1,3) -- (4,3);
    \draw[semithick, bg, |-|] (5,3) -- (6,3);
    \draw[semithick, bg, |-|] (1,2) -- (3,2);
    \draw[semithick, bg, |-|] (4,2) -- (6,2);
    \draw[semithick, bg, |-|] (1,1) -- (6,1);

    \node[dot, bg] (a) at (1.5,4) {};
    \node[dot, bg] (b) at (3.5,4) {};
    \node[dot, bg] (c) at (5.5,4) {};
    \node[dot, bg] (d) at (2.5,3) {};
    \node[dot, bg] (e) at (5.5,3) {};
    \node[dot, bg] (f) at (2,2) {};
    \node[dot, bg] (g) at (5,2) {};
    \node[dot, bg] (h) at (3.5,1) {};
    
    \draw[bg, semithick, decorate,decoration={brace}]
    (-.25,.75) -- (-.25,4.25) node[midway,left] {$\Alpha$\,\,};
    \draw[fg, semithick, ->, >={[round]Stealth}]
    (1-0.25,4.5) -- node[midway, above] {Time} (6+0.25,4.5);

    \draw[semithick, ->, >={[round,sep]Stealth}] (a) to[bend left = 20] (b);
    \draw[semithick, ->, >={[round,sep]Stealth}] (b) to[bend left = 20] (c);
    \draw[semithick, ->, >={[round,sep]Stealth}] (d) to[bend left = 20] (e);
    \draw[semithick, ->, >={[round,sep]Stealth}] (f) to[bend left = 20] (g);
    
    \draw[semithick, ->, >={[round,sep]Stealth}] (a) -- (d);
    \draw[semithick, ->, >={[round,sep]Stealth}] (b) to[bend left=15] (e);
    \draw[semithick, ->, >={[round,sep]Stealth}] (b) -- (g);
    \draw[semithick, ->, >={[round,sep]Stealth}] (d) to[bend left=10] (h);
    \draw[semithick, ->, >={[round,sep]Stealth}] (d) -- (g);
    \draw[semithick, ->, >={[round,sep]Stealth}] (f) -- (d);
    \draw[semithick, ->, >={[round,sep]Stealth}] (f) -- (h);
    \draw[semithick, ->, >={[round,sep]Stealth}] (g) to[bend left=10] (h);
    \draw[semithick, ->, >={[round,sep]Stealth}] (g) -- (c);
    \draw[semithick, ->, >={[round,sep]Stealth}] (h) to[bend left=10] (d);
    \draw[semithick, ->, >={[round,sep]Stealth}] (h) to[bend left=10] (g);

    \begin{scope}[on background layer]
      \foreach \a/\al in {1/3,2/2,3/1,4/0}{
        \node[anchor=east, baseline, color=bg] at (0.5,\a) {$\alpha_{\al}$};
        \foreach \t in {1,2,...,6}{
          \node[dot, lightgray,  minimum width = 1mm] (\a-\t) at (\t,\a) {};
        }      
      }    
      \draw[lightgray, semithick, ->, >={[round,sep]Stealth}] (4-1) -- (3-2);
      \draw[lightgray, semithick, ->, >={[round,sep]Stealth}] (4-2) -- (3-3);
      \draw[lightgray, semithick, ->, >={[round,sep]Stealth}] (4-4) -- (3-6);
      \draw[lightgray, semithick, ->, >={[round,sep]Stealth}] (4-3) -- (2-4);
      \draw[lightgray, semithick, ->, >={[round,sep]Stealth}] (1-1) -- (3-2);
      \draw[lightgray, semithick, ->, >={[round,sep]Stealth}] (1-2) -- (3-3);
      \draw[lightgray, semithick, ->, >={[round,sep]Stealth}] (1-4) -- (2-5);
      \draw[lightgray, semithick, ->, >={[round,sep]Stealth}] (1-3) -- (2-4);
      \draw[lightgray, semithick, ->, >={[round,sep]Stealth}] (2-5) -- (4-6);
      \draw[lightgray, semithick, ->, >={[round,sep]Stealth}] (2-5) -- (1-6);
      \draw[lightgray, semithick, ->, >={[round,sep]Stealth}] (2-5) -- (2-6);
      \draw[lightgray, semithick, ->, >={[round,sep]Stealth}] (2-2) -- (3-4);
      \draw[lightgray, semithick, ->, >={[round,sep]Stealth}] (2-2) -- (1-3);
      \draw[lightgray, semithick, ->, >={[round,sep]Stealth}] (3-3) -- (1-4);
      \draw[lightgray, semithick, ->, >={[round,sep]Stealth}] (3-3) -- (2-5);
    \end{scope}
  \end{tikzpicture}
  \caption{A \emph{time-interval network} for the 
    network from Figure~\ref{figure:time-expanded-network}. The
    epochs and transfer arcs of the original network are shown in
    light gray. The intervals in red, and the edges of the time-interval network in black.}
  \label{figure:time-intervalled-network}
\end{SCfigure}
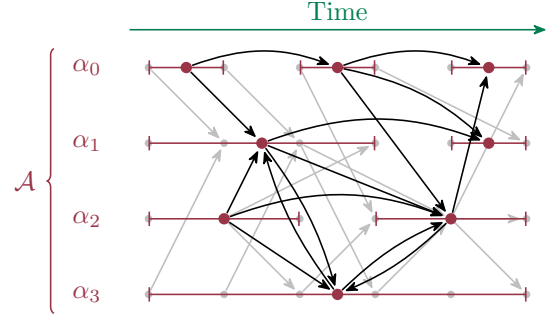

\noindent
We can use the \Lang{ilp} from Section~\ref{section:csp}
directly on time-interval networks rather than time-indexed
networks. Let $G$ be the time-expanded
network for a given \Lang{ktsp} instance, and let $I$ be a
time-interval network for an arbitrary partition of the time window of
each $\alpha\in\Alpha$. Let further
$\Gamma_{\Lang{ktsp}}$ and $\Gamma_I$ be the corresponding \Lang{ilp}s
initialized on these networks, respectively. Then we
have:

\begin{lemma}\label{lemma:interval-lb}
  The optimal solution for $\Gamma_I$ is a
  lower bound for the optimal solution of $\Gamma_\Lang{ktsp}$.
\end{lemma}
\begin{proof}
  Follows from the observation that for every solution in $\Gamma_{\Lang{ktsp}}$ there
  is a solution in $\Gamma_I$ with the same or a smaller value. The
  reason is that for any edge $\big((\alpha,t),(\beta,t')\big)$ in
  $G$ there are intervals $\lambda_{\alpha}\in\Lambda(\alpha)$ and
  $\lambda_{\beta}\in\Lambda(\beta)$ with $t\in\lambda_{\alpha}$ and
  $t'\in\lambda_{\beta}$ such that the edge
  $\big((\alpha,\lambda_{\alpha}),(\beta,\lambda_{\beta})\big)$ is
  present in $I$ and its weight is upper bounded by the weight of the
  edge $\big((\alpha,t),(\beta,t')\big)$.
\end{proof}

\noindent
Unfortunately, a solution of $\Gamma_I$ does not necessarily
correspond to a \emph{feasible} trajectory. The reason is
twofold. First, since we connect intervals by the best possible
transfer from any epoch within the interval, a trajectory may leave
an interval before it arrives, see the left side of
Figure~\ref{figure:ddd-glitches}. Second, since the time-interval
network is not acyclic, a solution to $\Gamma_I$ may contain
unconnected subtours (Figure~\ref{figure:ddd-glitches}, right).

\begin{figure}[h]
  \begin{tikzpicture}
    \draw[bg, semithick, |-|] (0, 0) -- (4,0);
    \draw[bg, semithick, |-|] (1,-1) -- (3,-1);
    \draw[bg, semithick, |-|] (3,-2) -- (6,-2);
    \node[bg, anchor=south] at (2,0)      {$\lambda_\alpha$};
    \node[bg, anchor=south] at (2.2,-1)   {$\lambda_\beta$};
    \node[bg, anchor=south] at (4.7,-1.9) {$\lambda_\gamma$};
    
    \node[anchor=east, baseline, color=bg] at (-0.25,0)  {$\alpha$};
    \node[anchor=east, baseline, color=bg] at (-0.25,-1) {$\beta$};
    \node[anchor=east, baseline, color=bg] at (-0.25,-2) {$\gamma$};
    
    \node[dot] (s) at (1, 0) {};
    \node[dot] (a) at (2,-1) {};
    \node[dot] (b) at (1.5,-1) {};
    \node[dot] (t) at (5,-2) {};

    \draw[semithick, ->, >={[round,sep]Stealth}] (s) -- (a);
    \draw[semithick, ->, >={[round,sep]Stealth}] (b) -- (t);
    
    \draw[fg, semithick, ->, >={[round]Stealth}]
    (0,0.7) -- node[midway, above] {Time} (6+0.25,0.7);
    
  \end{tikzpicture}
  \hfill
  \begin{tikzpicture}
    \draw[bg, semithick, |-|] (0, 0) -- (6,0);
    \draw[bg, semithick, |-|] (0,-1) -- (2.75,-1);
    \draw[bg, semithick, |-|] (3.25,-1) -- (6,-1);
    \draw[bg, semithick, |-|] (0,-2) -- (6,-2);
    \node[bg, anchor=south] at (3,0)      {$\lambda_\alpha$};
    \node[bg, anchor=south] at (3,-2) {$\lambda_\gamma$};
    
    \node[anchor=east, baseline, color=bg] at (-0.25,0)  {$\alpha$};
    \node[anchor=east, baseline, color=bg] at (-0.25,-1) {$\beta$};
    \node[anchor=east, baseline, color=bg] at (-0.25,-2) {$\gamma$};
    
    \node[dot] (s) at (0.5, -2)  {};
    \node[dot] (a) at (1.25,-1)  {};
    \node[dot] (b) at (2.5,0)    {};
    \node[dot] (c) at (4.2,-1)   {};
    \node[dot] (t) at (5,-2)     {};

    \draw[semithick, ->, >={[round,sep]Stealth}] (s) -- (a);
    \draw[semithick, ->, >={[round,sep]Stealth}] (a) -- (b);
    \draw[semithick, ->, >={[round,sep]Stealth}] (b) -- (c);
    \draw[semithick, ->, >={[round,sep]Stealth}] (c) -- (t);
    
    \draw[fg, semithick, ->, >={[round]Stealth}]
    (0,0.7) -- node[midway, above] {Time} (6+0.25,0.7);
    
  \end{tikzpicture}
  \caption{\textbf{Left:} Intervals $\lambda_\alpha$,
    $\lambda_\beta$, and $\lambda_\gamma$ corresponding to
    bodies $\alpha,\beta,\gamma$. The $x$-axis illustrates time, and black
    dots within the intervals indicate the time points used for the transfer between the intervals. A solution of~$\Gamma_I$
    can use both edges, while a solution of $\Gamma_{\Lang{ktsp}}$
    cannot. \textbf{Right:} A subtour in a time-interval network starting and ending at
    $\lambda_\gamma$. The \textbf{Flow Constraint} does \emph{not}
    exclude solutions that select this cycle, even if it
    is never reached from $\alpha_s$.}
  \label{figure:ddd-glitches}
\end{figure}
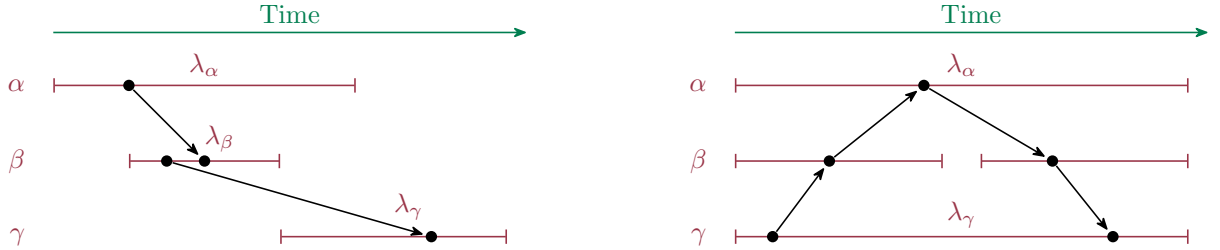

\subsection{Dynamic Generation of Intervals and Constraints}

Let us recap our findings so far: With $\Gamma_{\Lang{ktsp}}$ we have
a time-indexed formulation for \Lang{ktsp} that can find the optimum
via Observation~\ref{observation:time-expanded-opt}. With the
time-interval formulation $\Gamma_I$, we have a smaller \Lang{ilp}
that produces a lower bound via Lemma~\ref{lemma:interval-lb}, but
which may produce an \emph{infeasible} solution due to the issues
illustrated in Figure~\ref{figure:ddd-glitches}.  These considerations
lead to the ``dynamic'' in \emph{dynamic discretization discovery:} We
start with a very coarse interval partition of $\tgrid$ and solve the
corresponding $\Gamma_I$ (which hopefully is easy, since
the time-interval network is small). If the solution is
feasible, Lemma~\ref{lemma:interval-lb} implies it is optimal, and the
algorithm terminates. If the solution is \emph{not} feasible, it
either contains a temporal glitch or a subtour. In the first case, we
subdivide the violating interval (in Figure~\ref{figure:ddd-glitches}
this is $\lambda_\beta$) into smaller intervals (we ``discover
discretization'') and rerun the algorithm. In the second case, we add
a standard subtour-elimination constraint, for example the Dantzig,
Fulkerson, and Johnson inequality~\cite{DantzigFJ54}, and rerun the
algorithm. This strategy can be understood as simultaneously
applying \emph{column} and \emph{row}
generation~\cite{ChvatalCDFJ10}~--~new variables (i.e., columns) to
split intervals, and constraints (i.e., rows) to eliminate
subtours. Figure~\ref{figure:ddd} illustrates the algorithm and
the following theorem collects the insights established within this section:
\begin{theorem}
  The DDD algorithm presented in Figure~\ref{figure:ddd}
  computes an optimal solution of \Lang{ktsp} for $dt\rightarrow 0$.
\end{theorem}

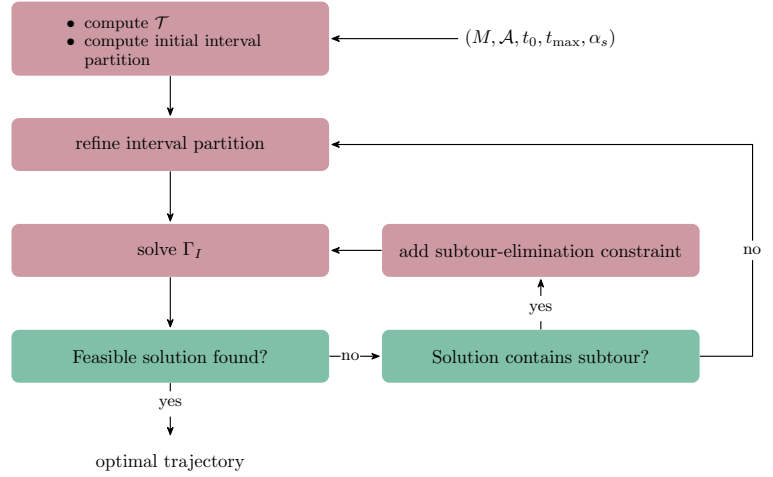
\begin{SCfigure}[2.0][htbp]
  \centering
  \scalebox{0.7}{\begin{tikzpicture}[
      box/.style = {
        inner sep = 0pt,
        minimum width=6cm,
        minimum height=1cm,
        baseline,
        anchor=west,
        fill=bg!50,
        rounded corners
      },
      io/.style = {
        fill = none,
        text centered
      },
      flow/.style = {
        ->,
        >={[sep,round]Stealth},
        semithick
      },
      decision/.style = {
        fill = fg!50,
      },
      result/.style = {
        circle,
        fill=white,
        inner sep=0pt,
        font=\small
      }
    ]

    \node[io] (input) at (10,-2)  {$(M,\Alpha,t_0,t_{\max},\alpha_s)$};
    \node[box, text width=5cm, minimum height=1.4cm] (init)  at (0,-2) {\raggedright\small\linespread{0.9}\selectfont%
      \begin{itemize}[noitemsep, nosep]
      \item compute $\tgrid$
      \item compute initial interval partition
      \end{itemize}%
    };
    \node[box]           (refine)    at (0,-4)  {refine interval partition};
    \node[box]           (solve)     at (0,-6)  {solve $\Gamma_I$};
    \node[box, decision] (done)      at (0,-8)  {Feasible solution found?};
    \node[box, io]       (output)    at (0,-10) {optimal trajectory};
    \node[box, decision] (subtour)   at (7,-8)  {Solution contains subtour?};
    \node[box]           (subtour-c) at (7,-6)  {add subtour-elimination constraint};

    \draw[flow] (input)     -- (init);
    \draw[flow] (init)      -- (refine);
    \draw[flow] (refine)    -- (solve);
    \draw[flow] (solve)     -- (done);
    \draw[flow,>={[round]Stealth}] (done)      -- node[pos=0.4, result] {yes} (output);
    \draw[flow] (done)      -- node[pos=0.4, result]  {no}  (subtour);
    \draw[flow] (subtour)   -- node[pos=0.4, result] {yes} (subtour-c);
    \draw[flow] (subtour-c) -- (solve);
    \draw[flow] (subtour)   -- ++(4,0) |- (refine);
    \node[result] at (14,-6) {no};
  \end{tikzpicture}}
  \caption{The DDD algorithm for \Lang{ktsp} as
    schematic flowchart. Red boxes describe processes, while green
    boxes define decisions. The concrete realization of the initial
    interval partition and the refinement step is not relevant, as
    long as the refinement partitions the interval that is
    participating in the conflict.}
  \label{figure:ddd}
\end{SCfigure}

%
%
\section{Improvements Based on Domain Knowledge}
\label{section:improvements}

Both the time-indexed formulation $\Gamma_{\Lang{ktsp}}$ and the
time-interval formulation $\Gamma_{I}$ can be handed to an
off-the-shelf \Lang{ilp} solver like Gurobi~\cite{gurobi}. The
technology behind these tools made impressive improvements in the last
decades~\cite{ClautiauxL25} and, as we will see in the next sections,
can solve many \Lang{ktsp} instances directly. However, if the instances grow larger, these
general-purpose solvers may reach their limits. There are various ways we can
support \Lang{ilp} solvers in such scenarios with domain knowledge:
\begin{description}
\item[Preprocessing] Based on our knowledge about the origin of the
  formulation, we can apply domain specific reduction rules to
  simplify the problem before we solve it.
\item[Constructive Heuristics] A heuristic may generate an initial
  solution, which can kick-start the general-purpose solver.
\item[Improving Heuristics] Another heuristic may improve intermediate
  solutions found by the \Lang{ilp} solver.
\item[Pre-computation of Constraints] In a dynamic setting like DDD,
  we may pre-compute some constraints from which we expect that they
  will be discovered anyway. 
\end{description}
\noindent
It is important to note that, even if these approach contain heuristic
elements, they do not alter the optimal solution of the underlying
problem. Hence, solving the \Lang{ilp}s with such domain-specific
improvements still guarantees to find the optimal trajectory. In this section, we
propose implementations of all four strategies for \Lang{ktsp}. Since,
the formal definition of \Lang{ktsp} and the \Lang{ilp}
formulations lie at the heart of this article,
we will only sketch the ideas within the main text and leave the
technical details to the appendix.

\subsection{Simplifying the Problem via Preprocessing}
\label{section:kernelization}

Reduction rules are a tool from fixed-parameter
theory~\cite{CyganFKLMPPS15} that obtain as input an instance of the
problem at hand (here, a time-expanded network $G$) and are
either \emph{applicable} or not. If they are applicable, they return
a solution-equivalent instance (a network $G'$) that is
``simpler,'' which usually just means smaller. We utilize the
following four rules, whereby $\textit{ub}$ is an upper bound on the
costs of an optimal trajectory that can be provided either by a
heuristic or by a domain expert.

\begin{reductionrule}[Heavy Arc Rule]\label{rule:heavy-arc}
  Applicable if there is an arc $e\in E(G)$ with $w(e)>\textit{ub}$. Delete $e$.
\end{reductionrule}
\begin{reductionrule}[Vee Rule]\label{rule:vee}
  Applicable if there is an arc $e=\big((\alpha,t),(\beta,t')\big)\in
  E(G)$ with $\beta\neq\alpha_s$ such that for
  all arcs $\tilde e=\big((\beta,t''),(\gamma,t''')\big)\in E(G)$ with
  $t''\geq t'$ and 
  $\gamma\neq\beta $ it
  holds that $w(e)+w(\tilde e)> \textit{ub}$. Delete $e$.
\end{reductionrule}
\begin{reductionrule}[Shortcut Rule]\label{rule:shortcut}
  Applicable if there is an arc $e=\big((\alpha,t),(\beta,t')\big)\in
  E(G)$ such that the
  $\mathrm{dist}\big((\alpha,t),(\beta,t')\big)<w(e)$. Delete $e$.
\end{reductionrule}
\begin{reductionrule}[Far Away Rule]\label{rule:far-away}
  Applicable if there is an arc $e=\big((\alpha,t),(\beta,t')\big)\in
  E(G)$ such that we have
  \(\mathrm{dist}\big((\alpha_s,t_0),(\alpha,t)\big)+w(e)+\mathrm{dist}\big((\beta,t'),(\alpha_s,t_{\max})\big)>\textit{ub}.\)
  Delete $e$.
\end{reductionrule}

The last three rules are illustrated in
Figure~\ref{figure:simplification rules} with an example. The
following theorem establishes the correctness of the rules and is formally
proven in Appendix~\ref{appendix:presolve}.

\noindent
\begin{theorem}
  \label{theorem:presolve}
  The reduction rules~\ref{rule:heavy-arc}--\ref{rule:far-away} are
  safe. Rule~\ref{rule:heavy-arc} can be applied exhaustivly in linear
  time, Rule~\ref{rule:vee} in time $O(|\Alpha|\cdot|\tgrid|)$, and
  rules~\ref{rule:shortcut} and~\ref{rule:far-away} in time $O(|V(G)|^3)$.
\end{theorem}
\begin{figure}[htb]
  \centering
  \begin{tabular}{l p{0.35\textwidth}@{}}
    \begin{tikzpicture}[
      label/.style = { midway, circle, fill=white, inner sep=0pt,
        font=\scriptsize\color{orange}}
    ]

    \foreach \t in {1,2,...,6}{
      \node[dot] (\t) at (\t,0) {};
    }      
    \foreach \t in {1,2,...,5}{
      \pgfmathtruncatemacro{\nextt}{\t+1}
      \draw[semithick, ->, >={[round,sep]Stealth}, color=gray] (\t) -- (\nextt);
    }

    \draw[semithick, <-, >={[round,sep]Stealth}] (1) -- node[label] {$1$} ++(-0.55,0.55);
    \draw[semithick, ->, >={[round,sep]Stealth}] (1) -- node[label] {$1$} ++(0.55,0.55);
    \draw[thick, color=bg, <-, >={[round,sep]Stealth}] (2) -- node[label] {$2$} ++(-0.4,0.8);
    \draw[semithick, ->, >={[round,sep]Stealth}] (2) -- node[label] {$4$} ++(0.45,0.8);
    \draw[semithick, ->, >={[round,sep]Stealth}] (2) -- node[label] {$6$} ++(0.75,0.5);
    \draw[thick, color=bg, ->, >={[round,sep]Stealth}] (3) -- node[label] {$3$} ++(0.6,0.6);
    \draw[semithick, ->, >={[round,sep]Stealth}] (4) -- node[label] {$5$} ++(0.5,0.75);
    \draw[semithick, ->, >={[round,sep]Stealth}] (4) -- node[label] {$4$} ++(0.75,0.45);
    \draw[semithick, ->, >={[round,sep]Stealth}] (5) -- node[label] {$7$} ++(0.6,0.6);
    
  \end{tikzpicture}  &
    \textbf{(a)} We can safely discard the incoming \textcolor{bg}{red} arc if $ub < 2+3 $. \\[8ex]
    \begin{tikzpicture}[
      label/.style = { midway, circle, fill=white, inner sep=0pt,
        font=\scriptsize\color{orange}}
    ]

    \foreach \a in {1,2,3,4}{
      \foreach \t in {1,2,...,6}{
        \node[dot] (\a-\t) at (\t,\a) {};
      }      
    }

    \foreach \a in {1,2,3,4}{
      \foreach \t in {1,2,...,5}{
        \pgfmathtruncatemacro{\nextt}{\t+1}
        \draw[semithick, ->, >={[round,sep]Stealth}, color=gray] (\a-\t) -- (\a-\nextt);
      }
    }

    \draw[thick, color=bg, ->, >={[round,sep]Stealth}] (4-1) -- node[label] {$18$} (3-6);

    \draw[thick, color=fg, ->, >={[round,sep]Stealth}] (4-1) -- node[label] {$2$} (2-2);
    \draw[thick, color=fg, ->, >={[round,sep]Stealth}] (2-2) to node[label] {$5$} (1-4);
    \draw[thick, color=fg, ->, >={[round,sep]Stealth}] (1-4) to node[label] {$0$} (1-5);
    \draw[thick, color=fg, ->, >={[round,sep]Stealth}] (1-5) to node[label] {$4$} (3-6);

    \draw[semithick, ->, >={[round,sep]Stealth}] (4-3) -- node[label] {$2$} (2-4);
    \draw[semithick, ->, >={[round,sep]Stealth}] (1-1) -- node[label] {$3$} (3-2);
    \draw[semithick, ->, >={[round,sep]Stealth}] (1-2) -- node[label] {$4$} (3-3);
    \draw[semithick, ->, >={[round,sep]Stealth}] (1-4) -- node[label] {$5$} (2-5);
    \draw[semithick, ->, >={[round,sep]Stealth}] (2-5) -- node[label] {$10$} (3-6);
    \draw[semithick, ->, >={[round,sep]Stealth}] (2-5) -- node[label] {$8$} (1-6);
    \draw[semithick, ->, >={[round,sep]Stealth}] (2-4) to[bend left] node[label] {$10$} (4-6);
    \draw[semithick, ->, >={[round,sep]Stealth}] (3-2) -- node[label] {$5$} (2-4);
  \end{tikzpicture} &
    \textbf{(b)} We drop the \textcolor{bg}{red} arc since an equivalent but cheaper \textcolor{fg}{path} exists. \\[8ex]
    \begin{tikzpicture}[
      label/.style = { midway, circle, fill=white, inner sep=0pt,
        font=\scriptsize\color{orange}}
    ]

    \foreach \a in {1,2,3,4}{
      \foreach \t in {1,2,...,9}{
        \node[dot] (\a-\t) at (\t,\a) {};
      }      
    }
    
    \foreach \a in {1,2,3,4}{
      \foreach \t in {1,2,...,8}{
        \pgfmathtruncatemacro{\nextt}{\t+1}
        \draw[semithick, ->, >={[round,sep]Stealth}, color=gray] (\a-\t) -- (\a-\nextt);
      }
    }
    
    \draw[thick, color=bg, ->, >={[round,sep]Stealth}] (3-3) to[bend left] node[label] {$5$} (1-5);
    \draw[thick, color=fg, ->, >={[round,sep]Stealth}] (4-1) -- node[label] {$2$} (3-2);
    \draw[thick, color=fg, ->, >={[round,sep]Stealth}] (3-2) -- node[label] {$0$} (3-3);
    \draw[thick, color=violet, ->, >={[round,sep]Stealth}] (1-5) to[bend left] node[label] {$3$} (3-7);
    \draw[thick, color=violet, ->, >={[round,sep]Stealth}] (3-7) to node[label] {$1$} (2-8);
    \draw[thick, color=violet, ->, >={[round,sep]Stealth}] (2-8) to  node[label] {$2$} (4-9);

    \draw[semithick, ->, >={[round,sep]Stealth}] (4-2) -- node[label] {$3$} (3-3);
    \draw[semithick, ->, >={[round,sep]Stealth}] (4-4) -- node[label] {$2$} (3-6);
    \draw[semithick, ->, >={[round,sep]Stealth}] (4-3) -- node[label] {$2$} (2-4);
    \draw[semithick, ->, >={[round,sep]Stealth}] (1-1) -- node[label] {$3$} (3-2);
     \draw[semithick, ->, >={[round,sep]Stealth}] (1-2) -- node[label] {$4$} (3-3);
     \draw[semithick, ->, >={[round,sep]Stealth}] (1-4) -- node[label] {$5$} (2-5);
     \draw[semithick, ->, >={[round,sep]Stealth}] (1-6) -- node[label] {$6$} (2-7);
     \draw[semithick, ->, >={[round,sep]Stealth}] (2-5) -- node[label] {$7$} (4-6);
     \draw[semithick, ->, >={[round,sep]Stealth}] (2-5) -- node[label] {$8$} (1-6);
     \draw[semithick, ->, >={[round,sep]Stealth}] (3-7) -- node[label] {$10$} (4-9);
     \draw[semithick, ->, >={[round,sep]Stealth}] (2-7) -- node[label] {$5$} (1-8);
     \draw[semithick, ->, >={[round,sep]Stealth}] (1-8) -- node[label] {$7$} (2-9);
     \draw[semithick, ->, >={[round,sep]Stealth}] (1-2) -- node[label] {$1$} (2-4); 

     \node[color=bg] at (8.8,4.4) {$\alpha_s$};

  \end{tikzpicture} &
    \textbf{(c)} Consider the \textcolor{bg}{red} arc. The \textcolor{fg}{shortest path to its
      tail} has cost~2 + 0, and the \textcolor{violet}{shortest path from its head} to
    the depot has cost~$3+1+2$. Hence, we can discard the red arc if $ub < 2+5+6=13$. 
  \end{tabular}
  \caption{Illustration of \textbf{(a)} the Vee Rule, \textbf{(b)} the
    Shortcut Rule, and \textbf{(c)} the Far Away Rule. }
  \label{figure:simplification rules}
\end{figure}
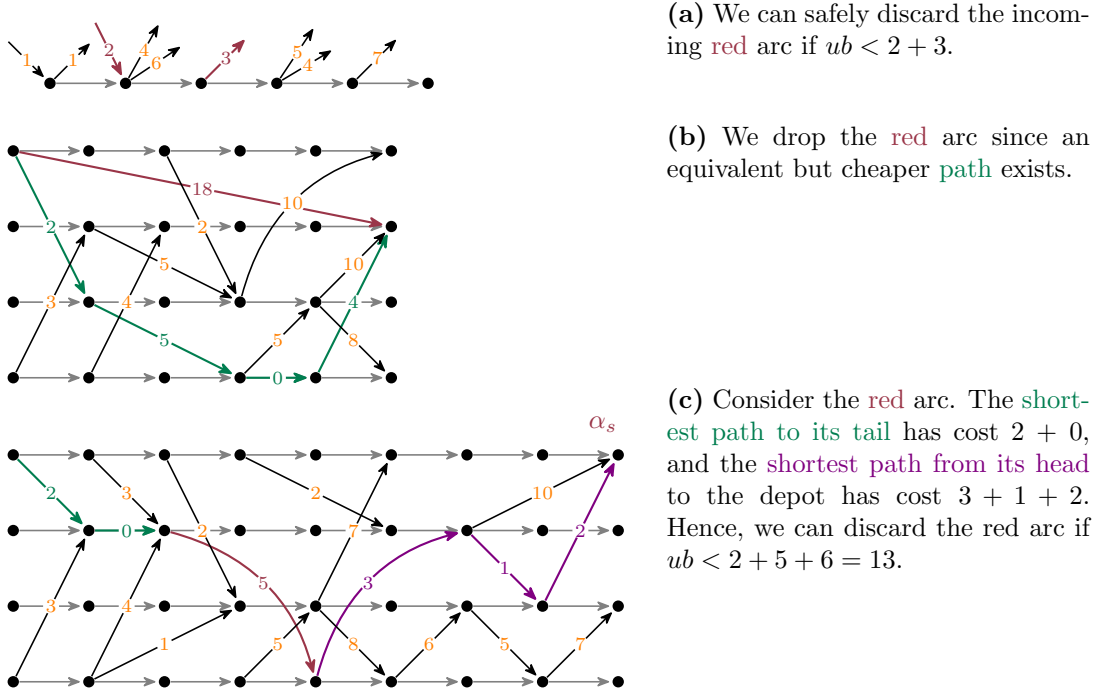

\subsection{A Constructive Heuristic to Find an Initial Solution}

In order to quickly generate an initial solution, we
use an adaptation of the \emph{insertion
heuristics} for the asymmetric traveling salesman problem with time
windows~\cite{AscheuerFG01}. 
For our setting, we introduce the \emph{permutation-schedule}
representation of trajectories. In
this representation, we encode a feasible trajectory
\begin{align*}
(\alpha_s,t_0)=(\alpha_{i_0},t_{i_0})
\rightarrow
(\alpha_{i_1},t_{i_1})
\rightarrow
\dots
\rightarrow
(\alpha_{i_{n-1}},t_{i_{n-1}})
\rightarrow
(\alpha_{i_{n}},t_{i_{n}})=(\alpha_s,t_{\max})
\end{align*}
in the form of two vectors
$\pi=(\alpha_{i_0},\alpha_{i_1},\dots,\alpha_{i_{n-1}},\alpha_{i_{n}})$
and $\sigma=(t_{i_0},t_{i_1},\dots,t_{i_{n-1}},t_{i_{n}})$, whereby we call $\pi$ the
\emph{permutation} and $\sigma$ the \emph{schedule}. Informally, $\pi$
describes in which order the trajectory visits the bodies, and the
schedule describes when each body is visited.

The idea of our \Algo{init} heuristic is to start with the
partial solution $\pi=(\alpha_s,\alpha_s)$ and
$\sigma=(t_0,t_{\max})$. Then, while there is some unvisited
$\alpha\in\Alpha$, we pick any such body and \emph{insert it} into
$\pi$ at some time $t$ in $\sigma$ that is locally optimal. The
details of this simple strategy are fleshed out in Appendix~\ref{appendix:init}.

\subsection{An Improving Heuristic to Enhance Intermediate Solutions}
\label{section:improving-heuristic}

Our improving heuristic, duped \Algo{swan} (for \emph{\underline{sw}ap
\underline{a}nd \underline{n}udge}), is an adaption of the famous 2-opt
heuristic~\cite{flood1956traveling,croes1958method} for the traveling
salesperson problem. Given an
initial solution in permutation-schedule representation, the \Algo{swan}
heuristic alternately fixes either $\pi$ or $\sigma$ and performs a
local search on the other:
\begin{description}
  \item[Swap Heuristic:] Considers $\sigma$ to be fixed and searches
    for two indices $i_{j}$ and $i_{k}$ such that swapping the two in
    $\pi$ results in a trajectory of lower cost.
  \item[Nudge Heuristic:] Considers $\pi$ to be fixed and searches
    for an index $i_{j}$ such that replacing $i_{j}$ in $\sigma$ by
    either $i_{j}-1$ or $i_{j}+1$ results in a better solution.
\end{description}
Conceptually, the swap heuristic is the 2-opt heuristic for a fixed
schedule, while the nudge heuristic tries to optimize locally the
epochs. We alternately perform both
heuristics exhaustively, i.e., we use the swap heuristic until there
is no possible improvement, then we use the shift heuristic until if
finds no further improvement; and we repeat this process until non of
the two finds any improvement at all.

A drawback of swap-based
heuristics like \Algo{swan} is that a run of the heuristic is not
unique in the following sense: Given $(\pi,\sigma)$, there can be two
or more pairs in $\pi$ (or points in $\sigma$) whose switch (or nudge)
improves the solution; these operations can exclude each other (for
instance, if they share a common element). We tackle this issue via a
metaheuristic, duped \Algo{b-swan}, that implements a beam search on
top of the search space of swaps and nudges.  The implementation
details are described in Appendix~\ref{appendix:improve}.

\subsection{Pre-Computing Subtour-Elimination Constraints}

The final improvement we utilizes is a standard trick to reduce the
number of iterations in dynamic algorithm like DDD: We pre-compute
some of the constraints that would otherwise be discovered by the
algorithm over time. A common choice in the realm of \Lang{tsp} is to
pre-compute the subtour-elimination constraints for all cycles of size
at most three~\cite{PferschyS17}. We will apply the same choice
in order to reduce the number of iterations of the inner-loop of DDD
(see Figure~\ref{figure:ddd}).

%
%
\section{A Benchmarkset for the Keplerian TSP}
\label{sec:benchmarkset}

We constructed and openly released a benchmark set of diverse
\Lang{ktsp} instances~\cite{dataset}, whereby each instance is
specified by a time window~$[t_0, t_f]$, the central body
parameter~$\mu$, and a list of orbiting bodies defined by their
Cartesian state at epoch $t_0$. To ensure accessibility and long-term
usability, we adopt a format inspired by the DIMACS standard (Center
for Discrete Mathematics and Theoretical Computer Science) and
distribute the data through the CERN-hosted Zenodo
platform~\cite{zenodo}
(\url{https://zenodo.org/records/14850862}). For every instance we
also provide several time-expanded formulations, which we provide in a
DIMACS-like format that includes all pre-computed transfer costs.
Although the exact procedure used to select the bodies in $\Alpha$ for
each proposed \Lang{ktsp} instance does not need to be explained in
detail here, it is useful to outline the main guiding principles. We
considered two representative orbital environments as sources: the
asteroid belt and the Jovian system.  Asteroid-belt instances include
selections of $|\Alpha| \in \{5, 10, 20, 40, 80\}$ bodies identified
as being well phased at some epoch $t \in [t_0, t_{\max}]$. This
ensures the presence of transfer opportunities with relatively low
cost, leading to solutions that resemble feasible interplanetary
trajectories in realistic preliminary mission design.  Jovian-moon
instances, by contrast, are constructed with $|\Alpha| \in \{4, 10,
20\}$ satellites chosen to have comparable semi-major axes, thereby
creating problem structures aligned with the dynamics of clustered
orbital systems. For all asteroid-belt instances, the width of the
time window $t_{\max} - t_0$ was fixed at 5.48~years, allowing, on
average, one and a half orbital revolutions to complete the tour. For
the Jovian-moon instances, a width of 6.57~years was adopted, enabling
the outer moons to complete multiple revolutions. The only exception
is the smallest instance, which includes solely the Galilean moons; in
this case, the time horizon is reduced significantly to 0.16~years,
since the outermost moon, Callisto, completes a single orbit in
16.7~days.

We utilize the following property of time-expanded networks to ensure
that optimal solutions with increasing $|\tgrid|$ form a decreasing
sequence.  Accordingly, we provide instances with 6,
11, 21, 41, 81, and 161 time points, enforcing the strict requirement $|\tgrid| >
|\Alpha|$. 

\begin{lemma}[Node Doubling]
\label{lemma:nodedoubling}
Any time-expanded instance of a \Lang{ktsp} defined over $|\Alpha|$ and $|\tgrid|$ has a solution that is strictly dominated by the corresponding instance with $|\tgrid'| = 2(|\tgrid|-1)+1$.
\end{lemma}
\begin{proof}
By construction, the grid spacing in the refined instance satisfies
$dt' = \tfrac{1}{2} dt$. Hence, every feasible solution in the original
instance~--~including the optimal one~--~remains feasible in the
refined instance. Consequently, the optimal cost of the refined
instance can only improve compared to the original.
\end{proof}

\subsection{Solutions of Selected Instances}
\label{section:besttours}
In this section, we briefly discuss the best trajectories found for
each of the instances with a discretization of 161 time points (which
is the highest we considered for the benchmark set).
Figure~\ref{fig:benchmarkset_plots} displays these trajectories along
with the average $\Delta V$ cost for each arc. For instances where the
global optimum (in the time-expanded network) was not reached, the
best $\Delta V$ obtained is indicated with an asterisk for clarity.

\begin{figure}[htb] 
  \includegraphics[width=\textwidth]{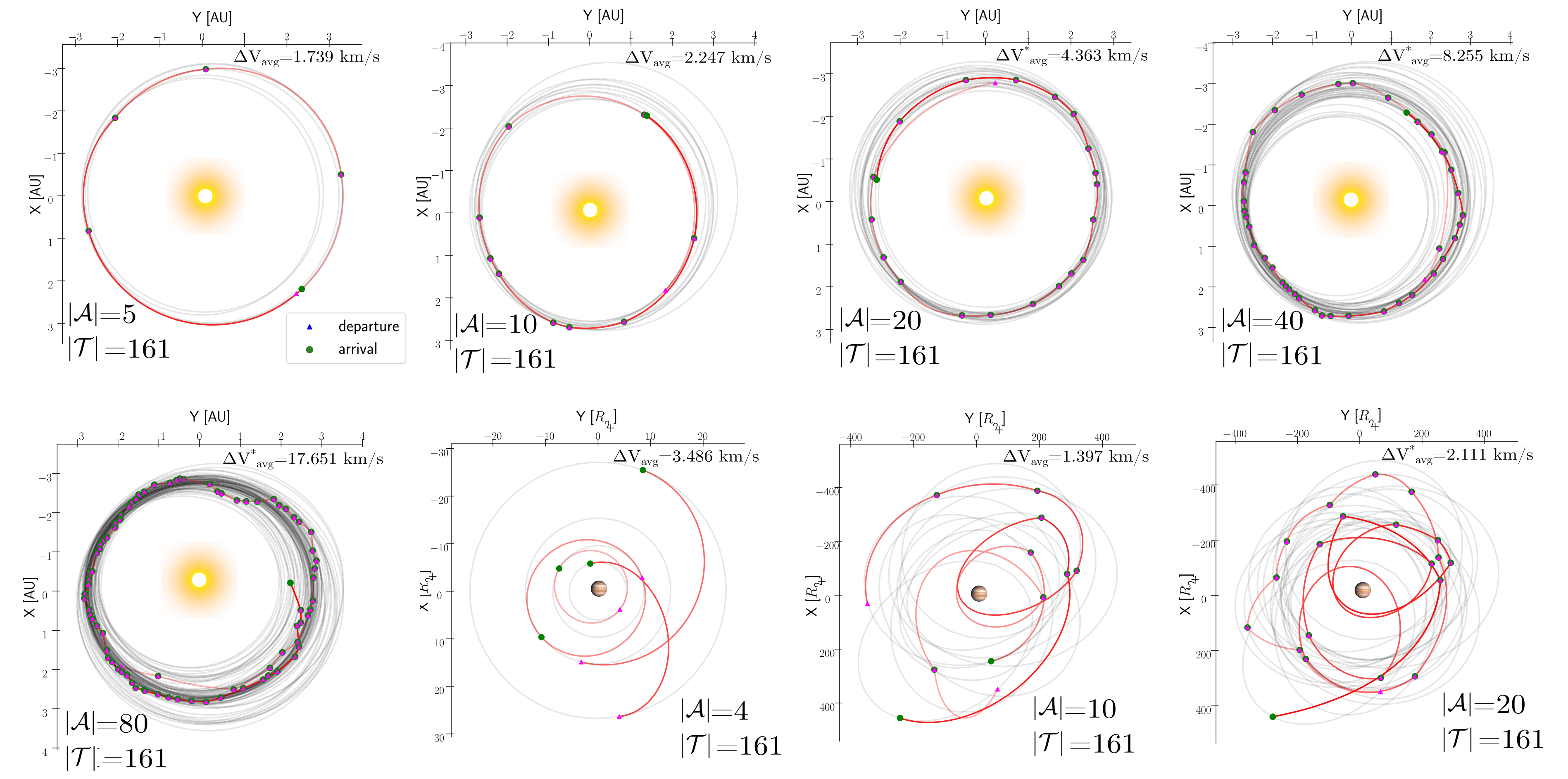} 
  \caption{Sample solutions from the benchmark set for both the asteroid belt and Jovian moons cases. Blue triangles mark departures, green dots arrivals, and red arcs represent Lambert transfers between orbits (grey). Reported values show the average $\Delta V$, with asterisks denoting upper bounds and not verified exact global minima.}
  \label{fig:benchmarkset_plots} 
\end{figure}

The two problem families exemplify contrasting dynamical regimes. In
the asteroid belt, bodies are well phased with similar semi-major
axes, so optimal tours consist of short arcs approximating a
continuous orbit. The tight 5.67~year time window used to create the
various instances prevents coasting to be useful, and as the number of
bodies rises to twenty, the average $\Delta V$ per transfer increases
sharply. 

The Jovian moons exhibit significant variations in
inclination and eccentricity, resulting in a
markedly different dynamical regime. The wide range of semi-major
axes produces orbital periods spanning from months to years, so the
time-expanded network captures the moons' positions at varying
coarseness, often leading to transfers presenting visible
discontinuities in velocities. As a consequence, the resulting trajectories
are costly and far from what a real mission could fly.  The
only exception is the instance containing only the Galilean moons,
where the time window and temporal resolution allow the optimal
solution in the time-expanded network to
consist of transfers very close to a Hohmann solution.

%
%
\section{Experimental Evaluation}
\label{sec:experiments}

We evaluate the techniques developed in this
article using the benchmark set introduced previously. We begin by
analyzing the effectiveness of the reduction rules from
Section~\ref{section:kernelization}, focusing on the percentage of
arcs they eliminate from the time-expanded network. Next, in
Section~\ref{section:plain-cop}, we assess the performance of the
\Lang{ilp} formulation from Section~\ref{section:csp}, comparing
various \Lang{ilp} solvers as backends. Section~\ref{section:ddd-experiment} presents an evaluation
of the dynamic discretization discovery method introduced in
Section~\ref{section:ddd}. Finally,  we compare the \Algo{b-swan}
heuristic with beam search and examine its impact when integrated into
the pipeline.

All experiments presented in this section were run on a server
equipped with two \verb|AMD EPYC 7702| 64-core processors operating at
a maximum clock speed of \verb|3353.5149 MHz|. The server has a total
of \verb|504 GB| of RAM and runs \verb|Arch Linux| with kernel version
\verb|6.13.5-arch1-1|.

\subsection{Performance of the Reduction Rules}

We run the reduction rules from Section~\ref{section:kernelization} on
all the time-expanded
networks. Figure~\ref{result:reduction-rule-performance} shows for
each rule the average amount of arcs it removes from the network if
applied without the other rules, as well as the performance of
applying all rules simultaneously.

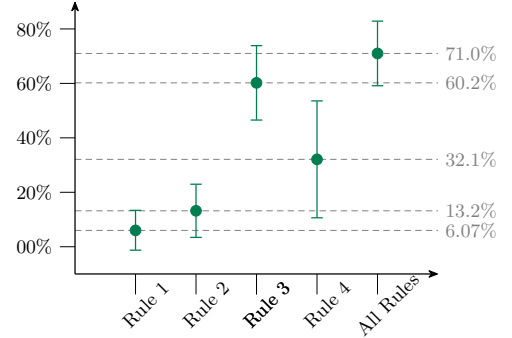
\begin{SCfigure}[2.0][htbp]
  \scalebox{0.8}{\begin{tikzpicture}[
      yscale = 0.45,
      dot/.style = {
        color = fg,
        circle,
        draw, fill,
        inner sep = 0pt,
        minimum width = 5pt
      },
      helper/.style = {
        color = gray,
        densely dashed
      }
    ]

    \draw[helper] (0,  0.6) -- (6,  0.6) node[right, baseline] {\small $6.07\%$};
    \draw[helper] (0, 1.32) -- (6, 1.32) node[right, baseline] {\small $13.2\%$};
    \draw[helper] (0, 6.02) -- (6, 6.02) node[right, baseline] {\small $60.2\%$};
    \draw[helper] (0, 3.21) -- (6, 3.21) node[right, baseline] {\small $32.1\%$};
    \draw[helper] (0, 7.10) -- (6, 7.10) node[right, baseline] {\small $71.0\%$};

    \foreach \x/\rule in {1/Rule 1, 2/Rule 2, 3/Rule 3, 3/Rule 3, 4/Rule 4, 5/All Rules} {
      \draw (\x, -1) -- ++(0, -.75) node[below, rotate = 45] {\small\rule};
    }
    \foreach \y in {0,2,4,...,8} {
      \draw (0, \y) -- (-0.25, \y) node[left, baseline] {\small$\y0\%$};
    }
    
    \draw[semithick, ->, >={[round]Stealth}] (0, -1) -- (0, 9);
    \draw[semithick, ->, >={[round]Stealth}] (0, -1) -- (6, -1);

    \draw[semithick, color = fg, |-|] (1, -0.15) -- (1, 1.36);
    \node[dot] at (1, 0.6) {};
    
    \draw[semithick, color = fg, |-|] (2, 0.32) -- (2, 2.32);
    \node[dot] at (2, 1.32) {};

    \draw[semithick, color = fg, |-|] (3, 4.63) -- (3, 7.41);
    \node[dot] at (3, 6.02) {};

    \draw[semithick, color = fg, |-|] (4, 1.04) -- (4, 5.38);
    \node[dot] at (4, 3.21) {};

    \draw[semithick, color = fg, |-|] (5, 5.89) -- (5, 8.31);
    \node[dot] at (5, 7.10) {};    
  \end{tikzpicture}}
  \caption{Average number of arcs removed from time-expanded networks
    by each reduction rule when applied individually, and by all rules
    applied together. Error bars represent the standard deviation
    across the benchmark set.}
  \label{result:reduction-rule-performance}
\end{SCfigure}

\subsection{Evaluation of Various ILP Solvers}
\label{section:plain-cop}

We evaluate the performance of different \Lang{ilp} solvers used
as backends for the formulation introduced in
Section~\ref{section:csp}. As shown in
Table~\ref{table:cop-comparison}, the commercial solver Gurobi
consistently outperforms the other tools on this dataset when paired
with the proposed encoding. Due to its superior performance, we
restrict the backend to Gurobi in all subsequent experiments.

\begin{SCtable}[1.0][htbp]
  \caption{Experimental results on a subset of the \Lang{ktsp}
    instances from the asteroid belt using different \Lang{ilp} solvers as backend for the
    \Lang{ilp} formulation. We report the computational time (in
    seconds) and the value of the optimal solution (OPT). The best time on each
    instance is \textbf{bold}.}
  \label{table:cop-comparison}
  \scriptsize\centering
  \linespread{1}\selectfont
\begin{tabular}{cc|ccccr} 
  \toprule
  %
  \multicolumn{2}{c|}{\emph{Instance (belt)}} &
  Gurobi & 
  HiGHS  & 
  Scip   & 
  CBC    &
  OPT\\
  %
  $|\Alpha|$ & $|\tgrid|$ &
  \hbox to 0pt{\emph{Solution time in seconds.}\hss}& 
  & 
  & 
  & 
  \\[0.25ex]
  \cmidrule(rl){1-7}
  %
  5 & 6 & 
  0.083 & 
  0.088 & 
  0.061 & 
  \textbf{0.055} & 
  12318.12 
  \\
  %
  %
  %
  5 & 11 & 
  0.113 & 
  0.239 & 
  0.250 & 
  \textbf{0.096} & 
  10938.92
  \\
  %
  %
  5 & 21 & 
  \textbf{0.181} & 
  1.369 & 
  0.742 & 
  0.233 & 
  10533.94
  \\
  %
  %
  5 & 41 & 
  \textbf{0.378} & 
  3.077 & 
  3.519 & 
  1.840 & 
  10459.16
  \\
  \cmidrule(rl){1-7}
  %
  %
  10 & 11 & 
  0.269 & 
  2.177 & 
  2.239 & 
  \textbf{0.242} & 
  38734.94
  \\
  %
  %
  10 & 21 & 
  \textbf{0.646} & 
  4.891 & 
  11.15 & 
  3.120 & 
  28173.62
  \\
  %
  %
  10 & 41 & 
  \textbf{2.136} & 
  22.49 & 
  85.25 & 
  26.24 & 
  25290.47
  \\
  \cmidrule(rl){1-7}
  %
  %
  20 & 21 & 
  \textbf{6.767} & 
  45.64 & 
  361.1 & 
  60.22 & 
  118104.77
  \\
    %
  %
  20 & 41 & 
  \textbf{354.6} & 
  37644.09 & 
  92137.19 & 
  20373.32 & 
  95033.47
  \\
  \bottomrule
\end{tabular}

\end{SCtable}

\begin{table}[htbp]
  \center%
  \caption{\linespread{1}\selectfont The impact of Dynamic Discretization Discovery on the size of
    the \Lang{ilp} formulation for some instances. The ``Plain'' columns refer to
    the encoding as presented in Section~\ref{section:csp}, the four
    columns under ``DDD'' present the \emph{largest} encountered
    instance during a run of the algorithm presented in
    Section~\ref{section:ddd} on this instance. The last two entries
    highlight the reduction of the encoding in percent and the number
    of \Lang{ilp} problems that needed to be solved. The last four
    columns use DDD with \emph{pre-computed cycle
    constraints} for cycles of size at most three. This modification
    increases the number of constraints, but reduces the number of
    \Lang{ilp} instances that must be solved.}
  \label{table:ddd-reduction}
  \resizebox{\textwidth}{!}{\linespread{1}\selectfont
\begin{tabular}{cc|cc|cccc|cccc}
  \toprule
  \multicolumn{2}{c|}{\emph{Instance (belt)}} & 
  Plain           &  &
  DDD             & & & &
  DDD-pc          & & & 
  \\
  $|\Alpha|$ & $|\tgrid|$ &
  \# Variables & \# Constraints &
  \# Variables & \# Constraints & Reduction & Calls &
  \# Variables & \# Constraints & Reduction & Calls    
  \\
  \cmidrule(rl){1-12}
  %
  %
  5   & 6   &
  277 & 34  &
  277 & 34  & 0\%    & 1 &
  277 & 34  & 0\% & 1 
  \\
  %
  %
  5   & 11  &
  984 & 59  &
  498 & 39  & \textbf{48.5\%} & 6 &
  498 & 39  & \textbf{48.5\%} & 6
  \\
  %
  %
  5    & 21   &
  3529 & 109  &
  2014 & 76   & \textbf{42.5\%} & 13 &
  2014 & 330  & \textbf{35.5\%} & 11
  \\
  \cmidrule(rl){1-12}
  %
  %
  10   & 11   &
  4099 & 119  &
  4099 & 119  & 0\% & 1 &
  4099 & 119  & 0\% & 1
  \\
  %
  %
  10     & 21 &
  15483  & 219 &
  8075   & 170 & \textbf{47.4\%} & 163 &
  8075   & 230 & \textbf{47.1\%} & 38
  \\
  %
  %
  10    & 41  &
  59212 & 419 &
  20438 & 247 & \textbf{65.3\%} & 242 &
  20438 & 921 & \textbf{64.1\%} & 13
  \\
  %
  %
  %
  %
  %
  %
  \bottomrule
\end{tabular}
}
\end{table}

\clearpage
\subsection{Impact of Dynamic Discretization Discover}
\label{section:ddd-experiment}

To evaluate the effectiveness of Dynamic Discretization Discovery
(DDD), we compare the size of the resulting \Lang{ilp} formulations
against the plain encoding introduced in
Section~\ref{section:csp}. While the plain encoding uses a fixed
discretization strategy, DDD adapts the discretization dynamically
during the solving process, potentially reducing the problem size
significantly. Table~\ref{table:ddd-reduction} summarizes the impact of this approach
across several benchmark instances. It highlights not only the
reduction in encoding size but also the trade-off between constraint
complexity and the number of \Lang{ilp} problems that need to be solved. 

While the reduction in encoding size is very promising, our current
implementation of DDD tends to converge to instances that appear
significantly more challenging for the
solver. Figure~\ref{figure:ddd-vs-direct} compares a run of the plain
encoding with a run of the DDD strategy. It can be observed that,
despite the smaller size of the final instances produced by DDD, they
are actually harder to solve than the original encoding. We therefore
conclude that, in its current form, DDD is not superior to the direct
approach. Nonetheless, the strategy remains promising~--~not only due to
the substantial reduction in encoding size, but also because related
work has shown that alternative backends such as \Lang{max-sat} can
outperform commercial solvers like Gurobi in similar
settings~\cite{CroellaLMV24}. This advantage stems from the fact that
\Lang{sat}-based approaches tend to handle scenarios
involving a large number of solver invocations slightly better. 

\begin{figure}[htbp]
  \centering
  \includegraphics[width=0.45\textwidth]{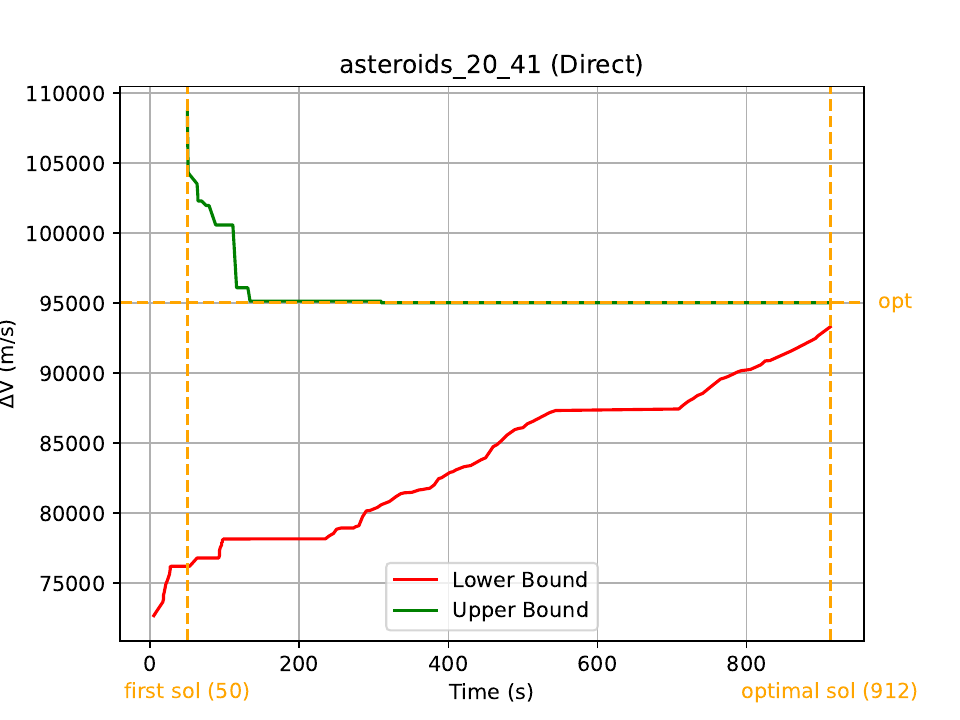}
  \quad
  \includegraphics[width=0.45\textwidth]{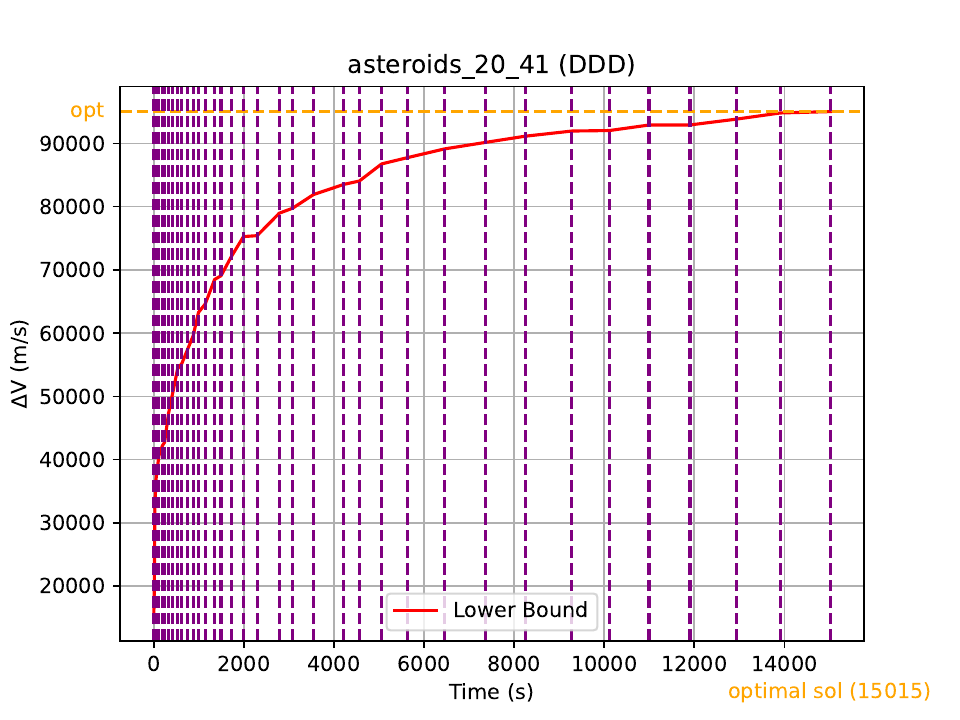}
  \caption{Comparison of the direct encoding of
    Section~\ref{section:csp} with the dynamic discretization encoding
  from Section~\ref{section:ddd} on the instance from the asteroid
  belt with $20$ asteroids and $81$ time points. Left the performance of the direct
  encoding that shows how the lower and upper bounds evolve over
  time. Right shows the DDD approach, whereby every purple
  lines indicates the start of a new round. This approach does not
  produce an upper bound, since it iteratively improves the lower bound.}
  \label{figure:ddd-vs-direct}
\end{figure}

\subsection{Assessment of the B-Swan Heuristic}
\label{section:swan-experiments}

In this section, we evaluate the performance of the \Algo{b-swan}
heuristic introduced in Section~\ref{section:improving-heuristic}. To
contextualize its effectiveness, we compare it to a classical beam search
with beam width of 1000 and 10000 (and without any heuristics improvements). The
evaluation is conducted across all instances from the benchmark set
described in
Section~\ref{sec:benchmarkset}. Table~\ref{table:heuristics}
presents the resulting optimality gaps, computed with respect to the
best lower or upper bounds we know for the these instances.

To investigate the integration of preprocessing and heuristic
techniques into the \Lang{ilp} pipeline, we now evaluate the impact of
combining the reduction rules and the \Algo{b-swan} heuristic with
Gurobi~--~identified as the most effective backend in
Table~\ref{table:cop-comparison}. Table~\ref{table:gurobi} in the
appendix presents
the results of running Gurobi with the encoding from
Section~\ref{section:csp} across all benchmark instances described in
Section~\ref{sec:benchmarkset}. We report the lower bounds (LB), upper
bounds (UB), and the resulting optimality gaps. 

One takeaway from the table is that preprocessing always
has a positive impact. Providing an initial
solution to Gurobi or assisting it by improving intermediate solutions
can slightly reduce performance on easier instances. This is expected,
as these strategies introduce overhead in cases where there is little
room for improvement. However, their benefits become evident on more
challenging instances~--~for example, the asteroid belt scenario with 40
bodies and 81 time points. In this case, the plain Gurobi approach
fails to find any feasible solution within the 4-hour time limit,
whereas enabling preprocessing allows a valid bound to be found.  

\begin{table}[ht]
  \caption{Comparison of the heuristic developed in
    this article against a beam search with
    beamwidth 10000 and 1000 on all instances from the benchmarkset. The part of the table
    before the double line contains the instance of the asteroid belt,
    while below are the once from the Jovian system. The gaps are computed with
    respect to the best lower bound or upper bound found in
    Table~\ref{table:gurobi} and highlighted, if they are best within
    the present table.}
  \label{table:heuristics}
  \resizebox{\textwidth}{!}{\begin{tabular}{@{}ll|lllr|lllr|lllr@{}}
\toprule
   &     & Beam-10.000                       &                &                & \multicolumn{1}{l|}{}         & Beam-1000                         &                &                & \multicolumn{1}{l|}{}         & B-Swan              &                &                & \multicolumn{1}{l}{}         \\
$|\Alpha|$  & $|\tgrid|$   & UB                                & Gap (VLB)      & Gap (VUB)      & \multicolumn{1}{l|}{Time (s)} & UB                                & Gap (VLB)      & Gap (VUB)      & \multicolumn{1}{l|}{Time (s)} & UB                & Gap (VLB)      & Gap (VUB)      & \multicolumn{1}{l}{Time (s)} \\ \midrule
\multicolumn{14}{c}{Asteroid-belt instances} \\  \midrule
5  & 6   & \textbf{12318.12}                 & \textbf{0.00}  & \textbf{0.00}  & 172.481µs                     & \textbf{12318.12}                 & \textbf{0.00}  & \textbf{0.00}  & 177.381µs                     & \textbf{12318.12} & \textbf{0.00}  & \textbf{0.00}  & 83.291µs                     \\
5  & 11  & \textbf{10938.91}                 & \textbf{0.00}  & \textbf{0.00}  & 285.951µs                     & \textbf{10938.91}                 & \textbf{0.00}  & \textbf{0.00}  & 275.952µs                     & \textbf{10938.91} & \textbf{0.00}  & \textbf{0.00}  & 0.003                        \\
5  & 21  & \textbf{10533.94}                 & \textbf{0.00}  & \textbf{0.00}  & 0.003                         & \textbf{10533.94}                 & \textbf{0.00}  & \textbf{0.00}  & 0.002                         & \textbf{10533.94} & \textbf{0.00}  & \textbf{0.00}  & 0.014                        \\
5  & 41  & \textbf{10459.16}                 & \textbf{0.00}  & \textbf{0.00}  & 0.016                         & \textbf{10459.16}                 & \textbf{0.00}  & \textbf{0.00}  & 0.005                         & 10503.28          & 0.42           & 0.42           & 0.016                        \\
5  & 81  & 10506.49                          & 0.71           & 0.71           & 0.1                           & 12068.34                          & 13.56          & 13.56          & 0.016                         & \textbf{10431.62} & \textbf{0.00}  & \textbf{0.00}  & 0.03                         \\
5  & 161 & 12987.41                          & 19.68          & 19.68          & 0.3                           & 38034.71                          & 72.57          & 72.57          & 0.026                         & 10834.68          & \textbf{3.72}  & \textbf{3.72}  & 0.058                        \\ \midrule
10 & 11  & 41437.79                          & 6.52           & 6.52           & 231.231µs                     & 41437.79                          & 6.52           & 6.52           & 201.421µs                     & \textbf{38734.93} & \textbf{0.00}  & \textbf{0.00}  & 0.017                        \\
10 & 21  & 29847.05                          & 5.61           & 5.61           & 0.038                         & 29847.05                          & 5.61           & 5.61           & 0.005                         & 28254.54          & \textbf{0.29}  & \textbf{0.29}  & 0.006                        \\
10 & 41  & 35383.59                          & 28.52          & 28.52          & 0.1                           & 39274.27                          & 35.61          & 35.61          & 0.012                         & \textbf{25290.47} & \textbf{0.00}  & \textbf{0.00}  & 6.2                          \\
10 & 81  & 215231.76                         & 88.45          & 88.45          & 0.2                           & 272736.66                         & 90.89          & 90.89          & 0.027                         & 24965.96          & \textbf{0.46}  & \textbf{0.46}  & 0.18                         \\
10 & 161 & 263105.81                         & 90.61          & 90.61          & 0.3                           & 285558.47                         & 91.34          & 91.34          & 0.056                         & 24765.34          & \textbf{0.19}  & \textbf{0.19}  & 8.5                          \\ \midrule
20 & 21  & 200991.38                         & 41.24          & 41.24          & 198.651µs                     & 200991.38                         & 41.24          & 41.24          & 192.471µs                     & 121898.06         & \textbf{3.11}  & \textbf{3.11}  & 0.039                        \\
20 & 41  & 229067.11                         & 58.51          & 58.51          & 0.2                           & 229067.11                         & 58.51          & 58.51          & 0.028                         & 98109.02          & \textbf{3.13}  & \textbf{3.13}  & 7.2                          \\
20 & 81  & 445254.35                         & 81.82          & 79.22          & 0.6                           & 545685.15                         & 85.16          & 83.04          & 0.077                         & 95404.03          & \textbf{15.14} & \textbf{3.00}  & 4.6                          \\
20 & 161 & 1195351.29                        & 94.09          & 92.33          & 1.1                           & 979616.50                         & 92.79          & 90.65          & 0.09                          & 100353.12         & \textbf{29.62} & \textbf{8.69}  & 1.3                          \\ \midrule
40 & 41  & 999661.61                         & 60.88          & 60.88          & 458.573µs                     & 999661.61                         & 60.88          & 60.88          & 474.553µs                     & 450838.70         & \textbf{13.26} & \textbf{13.26} & 124                          \\
40 & 81  & {\color[HTML]{0C0D0E} 1943476.34} & 86.66          & 83.69          & 0.90                          & {\color[HTML]{0C0D0E} 1652302.56} & 84.31          & 80.82          & 0.10                          & 368538.01         & \textbf{29.67} & \textbf{14.01} & 15.80                        \\
40 & 161 & 3454397.64                        & 93.49          & 90.95          & 2.2                           & 3590681.86                        & 93.74          & 91.30          & 0.2                           & 366870.32         & \textbf{38.72} & \textbf{14.82} & 533                          \\ \midrule
80 & 81  & 2442965.08                        & 49.40          & 36.25          & 1.2                           & 2442965.08                        & 49.40          & 36.25          & 0.001                         & 1888709.97        & \textbf{34.55} & \textbf{17.55} & 707                          \\
80 & 161 & 5268435.39                        & 85.53          & 73.03          & 4                             & 5268544.46                        & 85.53          & 73.03          & 0.5                           & 1627976.24        & \textbf{53.17} & \textbf{12.73} & 7.8                          \\ \midrule
\multicolumn{14}{c}{Jovian-moon instances} \\  \midrule
4  & 6   & 42806.79                          & 9.58           & 9.58           & 183.981µs                     & 42806.79                          & 9.58           & 9.58           & 145.421µs                     & \textbf{38705.56} & \textbf{0.00}  & \textbf{0.00}  & 0.003                        \\
4  & 11  & 28112.39                          & 28.51          & 28.51          & 186.921µs                     & 28112.39                          & 28.51          & 28.51          & 194.101µs                     & 24771.83          & \textbf{18.87} & \textbf{18.87} & 0.003                        \\
4  & 21  & 21240.77                          & \textbf{12.35} & \textbf{12.35} & 216.482µs                     & 21240.77                          & \textbf{12.35} & \textbf{12.35} & 220.191µs                     & 24426.14          & 23.78          & 23.78          & 0.003                        \\
4  & 41  & 20260.10                          & \textbf{11.94} & \textbf{11.94} & 223.431µs                     & 20260.10                          & \textbf{11.94} & \textbf{11.94} & 481.952µs                     & 22575.57          & 20.97          & 20.97          & 0.003                        \\
4  & 81  & 18464.03                          & \textbf{4.55}  & \textbf{4.55}  & 0.003                         & 18464.03                          & \textbf{4.55}  & \textbf{4.55}  & 0.001                         & 26498.51          & 33.49          & 33.49          & 0.001                        \\
4  & 161 & 18248.64                          & \textbf{4.48}  & \textbf{4.48}  & 0.012                         & 18248.64                          & \textbf{4.48}  & \textbf{4.48}  & 0.004                         & 23505.26          & 25.84          & 25.84          & 0.037                        \\ \midrule
10 & 11  & 20381.28                          & 2.50           & 2.50           & 238.082µs                     & 20381.28                          & 2.50           & 2.50           & 214.291µs                     & \textbf{19871.59} & \textbf{0.00}  & \textbf{0.00}  & 0.7                          \\
10 & 21  & 17866.39                          & 8.94           & 8.94           & 0.029                         & 17866.39                          & 8.94           & 8.94           & 0.005                         & 17049.86          & \textbf{4.58}  & \textbf{4.58}  & 0.7                          \\
10 & 41  & 17360.86                          & 10.08          & 10.08          & 0.1                           & 18461.53                          & 15.44          & 15.44          & 0.019                         & 16126.06          & \textbf{3.19}  & \textbf{3.19}  & 1.6                          \\
10 & 81  & 22388.45                          & 31.02          & 31.02          & 0.4                           & 34905.86                          & 55.76          & 55.76          & 0.048                         & 15853.17          & \textbf{2.58}  & \textbf{2.58}  & 15.4                         \\
10 & 161 & 62147.67                          & 75.26          & 75.26          & 0.9                           & 83009.90                          & 81.48          & 81.48          & 0.1                           & 16066.03          & \textbf{4.31}  & \textbf{4.31}  & 0.3                          \\ \midrule
20 & 21  & 59865.54                          & 16.01          & 16.01          & 195.571µs                     & 59865.54                          & 16.01          & 16.01          & 180.411µs                     & 53839.95          & \textbf{6.61}  & \textbf{6.61}  & 1.2                          \\
20 & 41  & 65257.25                          & 28.74          & 28.74          & 0.2                           & 76954.06                          & 39.57          & 39.57          & 0.03                          & 48056.73          & \textbf{3.23}  & \textbf{3.23}  & 1.7                          \\
20 & 81  & 118213.20                         & 63.13          & 62.10          & 0.8                           & 128832.04                         & 66.17          & 65.23          & 0.1                           & 49309.55          & \textbf{11.60} & \textbf{9.15}  & 0.4                          \\
20 & 161 & 204549.12                         & 81.41          & 78.32          & 2.5                           & 298419.22                         & 87.26          & 85.14          & 0.3                           & 46681.22          & \textbf{18.54} & \textbf{5.01}  & 57.8                         \\ \bottomrule
\end{tabular}
}
\end{table}

%
%
\section{Conclusion}
\label{section:conclusion}

We introduced and studied the Keplerian traveling salesperson problem
as a problem at the interface of aerospace trajectory design and
discrete optimization, and we established the first openly available
benchmark set to foster progress on this problem. We argued how such a
formal problem is at the heart of all efforts to design multiple
rendezvous interplanetary missions. Our study shows that time-expanded
network formulations, combined with dedicated preprocessing and
heuristic improvements, provide a viable way of tackling the
combinatorial complexity of the task. At the same time, we
demonstrated how Dynamic Discretization Discovery can overcome the
explosion of the problem size inherent in time-indexed formulations,
additionally offering a natural bridge between continuous optimization
techniques familiar in astrodynamics and discrete optimization methods
developed in operations research.

Our computational study highlights both the promise and the current
limitations of relying on integer linear programming solvers for this
problem class. While state-of-the-art \Lang{ilp} technology proves
effective, the strong temporal structure of \Lang{ktsp} suggests that
alternative paradigms, such as \Lang{max-sat} and related
\Lang{sat}-based methods, may offer additional advantages,
particularly in dynamically evolving formulations. Exploring these
directions could unlock further synergies between aerospace
applications and modern combinatorial optimization technologies.

%
%
\section*{Acknowledgment}

Some of the ideas presented in this work were discussed and refined during the
Dagstuhl Seminar 25362, \emph{Optimization and Automated Reasoning for
Designing Future Space Missions,} held from August 31 to September 3,
2025. 

%
%
\clearpage
\section*{Bibliography}
\bibliography{main}

%
%
\clearpage
\section{Technical Appendix}
\label{app:details}

\subsection{Proof of Theorem~\ref{theorem:presolve}}
\label{appendix:presolve}

We prove the theorem in the form of four lemmas, each establishing the
correctness and running time of one of the reduction rules.

\begin{lemma}
  Rule~\ref{rule:heavy-arc} is safe and can be applied exhaustively in
  linear time. 
\end{lemma}
\begin{proof}
  It suffices to scan over the arcs of $G$ and filter out
  the ones with weight exceeding the bound, which is possible in time
  $O(|E(G)|)$. For correctness let $G$ be the input network and $G'$
  the network obtained by applying the rule. Then $\val(G')=\val(G)$
  since in $G$ there is a trajectory $P$ with
  $\val(P)=\val(G)\leq\mathit{ub}$, which is present in $G'$ since no
  arc in $P$ can have a weight of more than $\mathit{ub}$.
\end{proof}

\begin{lemma}
  Rule~\ref{rule:vee} is safe and can be applied exhaustively in time $O(|\Alpha|\cdot|\tgrid|)$. 
\end{lemma}
\begin{proof}
  For the correctness take a feasible trajectory $P$ with
  $\val(P)=\val(G)$. Assume $P$ uses the arc
  $e=\big((\alpha,t),(\beta,t')\big)$, then it needs to leave
  $\beta$ at some point $t"\geq t'$ via some arc $\tilde e$ since $\beta\neq\alpha_s$. But
  $P$ cannot contain any of these arcs since
  $\val(G)=\val(P)\leq\mathit{ub}<w(e)+w(\tilde e)$.

  For the runtime we manage a hash table $T$ that stores for every
  node $(\alpha,t)\in\Alpha\times\tgrid$ the cost of the cheapest arc
  leaving $\alpha$ at some point $t'\geq t$. We initialize $T$ by
  scanning over all arcs ones (in time $O(|E|)$) and store the
  cheapest arc leaving $(\alpha,t)$ (at time $t$). Finally, for every
  $\alpha\in\Alpha$ we start at $t_{\max}$ and propagate the minimum
  value in $T$ backwards (which requires time $O(|\Alpha|\cdot|\tgrid|)$).
\end{proof}

\begin{lemma}
  Rule~\ref{rule:shortcut} is safe and can be applied exhaustively in time $O(|V(G)|^3)$. 
\end{lemma}
\begin{proof}
  Correctness follows by a simple switching argument: Consider any
  feasible trajectory $P$ that uses~$e$, then we can obtain another
  feasible trajectory $P'$ by replacing $e$ with the shortest path
  between $(\alpha,t)$ and $(\beta,t')$. By construction,
  $\val(P')<\val(P)$ and, thus, $P$ can not be an optimal
  trajectory.

  For the runtime, we initially compute a distance matrix using
  any all-pair-shortest-path algorithm (say, Floyd–Warshall),
  which is possible in time $O(|V(G)|^3)$. Afterwards, we simply scan
  once over all edges in time $O(|E(G)|)$ and filter out the once that
  are too expensive.
\end{proof}

\begin{lemma}
  Rule~\ref{rule:far-away} is safe and can be applied exhaustively in time $O(|V(G)|^3)$. 
\end{lemma}
\begin{proof}
  For safety, observe that for any feasible trajectory $P$ that uses
  $e$ it holds:
  \[\val(P)\geq \mathrm{dist}\big((\alpha_s,t_0),(\alpha,t)\big)+w(e)+\mathrm{dist}\big((\beta,t'),(\alpha_s,t_{\max})\big).\]
  This is because in order to use $e$, the trajectory must
  reach $(\alpha,t)$ (which costs 
  $\mathrm{dist}\big((\alpha_s,t_0),(\alpha,t)\big)$) and then get from
  $(\beta,t')$ to the depot (which costs at least
  $\mathrm{dist}\big((\beta,t'),(\alpha_s,t_{\max})\big)$). Hence, we
  can safely delete $e$ since $P$ cannot be optimal because of
  $\val(P)>\mathit{ub}$.
  \noindent
  The runtime is equivalent to Rule~\ref{rule:shortcut}.
\end{proof}

\subsection{Details of the Constructive Heuristic}
\label{appendix:init}

The \Algo{init} heuristic shown in Figure~\ref{figure:constructive} is
an adaptation of the \emph{insertion heuristics} for the asymmetric
traveling salesman problem with time windows~\cite{AscheuerFG01}.  In
order to keep the description simple, we will assume in this section that all
transfers are feasible. We may artificially enforce this
requirement by assigning unreasonable high $\Delta V$ values to the
unwanted transfers.

\begin{figure}[htb]
  \begin{minipage}[t]{0.49\textwidth}
    \input{code/swan-init}
  \end{minipage}
  \begin{minipage}[t]{0.49\textwidth}
    \input{code/insert}    
  \end{minipage}
  \caption{The \Algo{init} heuristic that generates initial
    feasible trajectories by
    successively inserting elements to a growing trajectory.}
  \label{figure:constructive}
\end{figure}

\begin{lemma}
  Algorithm \Algo{init} runs in time
  $O(|\Alpha|^2\cdot|\tgrid|)$ and always outputs a feasible trajectory.
\end{lemma}
\begin{proof}
  For the correctness, recall that in this section we assume that $G$
  contains all possible transfers, i.e., $(\pi,\sigma)$ is feasible if
  $\pi$ is a permutation with the elements starting and ending with
  $\alpha_s$ and $\sigma$ is strictly increasing with $\sigma[0]=t_0$
  and $\sigma[n]=t_{\max}$. The former is ensured by lines~\ref{alog:swan-init:pi0}
  and~\ref{alog:swan-init:perm} in \Algo{init}, the latter by
  Line~\ref{alog:swan-init:sigma0} in \Algo{init} and
  Line~\ref{alog:insert:t} in \Algo{insert}.

  For the runtime note that \Algo{insert} is called $n-1 =
  |\Alpha|-1$ times in Line~\ref{alog:swan-init:insert}. Each call of
  \Algo{insert} on the other hand contains a \texttt{for}-loop over
  $|\Alpha|$ elements (Line~\ref{algo:insert:loop} in \Algo{insert}) and the choice of
  $t$ in Line~\ref{alog:insert:t}, which requires the inspection of
  $O(|\tgrid|)$ elements. Hence, \Algo{insert} runs in time
  $O(|\Alpha|\cdot|\tgrid|)$ and, therefore, \Algo{init} in
  time $O(|\Alpha|^2\cdot|\tgrid|)$.
\end{proof}

\subsection{Details of the Improving Heuristic}
\label{appendix:improve}

As in the previous section, we will assume for simplicity that all
transfers are feasible. Hence, every pair~$(\pi,\sigma)$ represents a
feasible trajectory $P$ and we can define
$\val(\pi,\sigma)\coloneq\val(P)$.  Let us further call a trajectory
\emph{2-opt} if the swap heuristic cannot be applied, \emph{nudge-opt}
if the nudge heuristic cannot be applied, and \emph{2-nudge-opt} if
the \Algo{swan} heuristic cannot be applied.

The \Algo{b-swan}
heuristics maintains a \emph{priority queue} $Q$ in the form of a
\emph{min-heap} using trajectories $(\pi,\sigma)$ as values and
$\val(\pi,\sigma)$ as keys. Hence, we can store newly discovered
trajectories in time $O(\log |Q|)$ and obtain the currently cheapest
trajectory in time $O(\log |Q|)$. The idea then is to initialize the
queue with the set $S$ of given trajectories. While the queue is not
empty, we pick the best $(\pi,\sigma)$ and analyze it:
\begin{itemize}
\item If $\pi$ is \emph{not} 2-opt, we add a trajectory $(\pi',\sigma)$ for
  every possible swap to the queue;
\item if $\pi$ \emph{is} 2-opt, but $\sigma$ is \emph{not} nudge-opt,
  we add $(\pi,\sigma')$ for every possible nudge to the
  queue;
\item if $\pi$ is 2-opt and $\sigma$ is nudge-opt and
  $\val(\pi,\sigma)$ is smaller than the currently best solution, we
  register $(\pi,\sigma)$ as a new solution and we add
  randomly perturbed trajectories $(\pi',\sigma')$ back to the heap.
\end{itemize}

Since the priority queue can grow quickly (whenever we remove an
element from the heap, we add back new elements unless we extracted a
2-nudge-opt trajectory that does not improve the current best
solution), we need a mechanism to reduce it again. We realize this by
adapting a beam search to this setting: Given a \emph{beamwidth} $w$
and a \emph{shrink-factor} $f$ as arguments, we will shrink the
priority queue to $fw$ elements once it contains more than $w$
elements. Figure~\ref{algorithm:swan} contains the details of the
\Algo{b-swan} metaheuristic, and the following lemma observes that
\Algo{b-swan} always returns a feasible trajectory in
Line~\ref{algo:swan:return}.

\begin{lemma}
  The \Algo{b-swan} metaheuristic always returns a feasible trajectory if
  $S$ contains only feasible trajectories.
\end{lemma}
\begin{proof}
  It is sufficient to show that all trajectories that are inserted
  into $Q$ are feasible. By assumption, this is the case in
  Line~\ref{algo:swan:init}. Given a feasible trajectory
  $(\pi,\sigma)$ with $\pi =
  (\alpha_s=\alpha_{i_0},\alpha_{i_1},\dots,\alpha_{i_n-1},\alpha_{i_n}=\alpha_s)$,
  any trajectory $(\pi',\sigma)$ with
  $\pi'=(\alpha_s=\alpha_{i_0},\text{perm}(\alpha_{i_1},\dots,\alpha_{i_n-1}),\alpha_{i_n}=\alpha_s)$
  is feasible since we assume $G$ to contain all possible transfers in
  this section, whereby
  $\text{perm}(\alpha_{i_1},\dots,\alpha_{i_n-1})$ denotes an arbitrary
  permutation of the elements
  $\alpha_{i_1},\dots,\alpha_{i_n-1}$. Hence, trajectories inserted in
  Line~\ref{algo:swan:insert-swap} and~\ref{algo:swan:insert-perturb}
  are feasible,
  as the corresponding subroutines permute only internal elements of
  $\pi$. On the other hand, trajectories inserted in
  Line~\ref{algo:swan:insert-nudge} are feasible since the
  \Algo{nudge} subroutine actively checks feasibility in Line~\ref{algo:nudge:feasability}.
\end{proof}

The precise runtime of the \Algo{b-swan} metaheuristic depends on the termination
criteria used in Line~\ref{algo:swan:termination}. Suitable options
are a fixed number of iterations, the fact that $Q$ runs empty, or that no new
solution was found for a certain amount of rounds. Independent of this
criteria, the running time of each iteration of the
\texttt{while}-loop can be bounded as follows:

\clearpage

\begin{lemma}
  Assuming $w\geq n=|\Alpha|$, a single iteration of the
  \texttt{while}-loop in Line~\ref{algo:swan:termination} runs in 
  $O(w^2\log w)$.
\end{lemma}
\begin{proof}
  We first inspect the maximum size the queue $Q$ can have when the
  algorithm starts a iteration of Line~\ref{algo:swan:termination}. By
  Line~\ref{algo:swan:shrink1} and~\ref{algo:swan:shrink2}, the queue
  is shrink to $fw$ elements if it contains more than $w$
  trajectories. Since in Line~\ref{algo:swan:insert-swap} we insert at
  most $n^2\in O(w^2)$ elements, in
  Line~\ref{algo:swan:insert-nudge} at most $n\in
  O(w)$ elements, and in Line~\ref{algo:swan:insert-perturb} at most
  $fw\in O(w)$ elements, the maximum size of the queue at any point of
  the algorithm is $O(w+w^2)=O(w^2)$. Hence, extracting and inserting
  elements to the queue requires at most $O(\log w^2)=O(\log w)$ time.

  It immediately follows that the \Algo{shrink-queue} subroutine
  takes time at most $O(fw\log w)$, the \Algo{swap} routine
  (including Line~\ref{algo:swan:insert-swap} in \Algo{b-swan}) at
  most $O(w^2\log w)$ steps, and both, 
  the \Algo{nudge} and the \Algo{perturb} algorithm
  (including the corresponding insert in \Algo{b-swan}) at
  most $O(w\log w)$ operations. Note that in all subroutines, we can
  check whether the new trajectories are
  feasible and compute their value in constant time, as we start from a
  feasible trajectory (for which we know its value) and only perform
  local modifications.
  Since only one of these subroutines is
  called and since $f\leq 1$, the dominating term for of the
  \texttt{while}-loop is is $O(w^2\log
  w)$. 
\end{proof}

\begin{figure}[htbp]
  \begin{minipage}[t]{0.45\textwidth}
    \input{code/swan}
  \end{minipage}
  \begin{minipage}[t]{0.54\textwidth}
    \input{code/shrink}
    
    \input{code/swaps}

    \input{code/nudges}

    \input{code/perturb}
  \end{minipage}
  \caption{The details of the \Algo{b-swan} heuristic. As input it
    obtains a time-expanded network $G$, a set $S$ of feasible
    trajectories in permutation-schedule representation, a beamwidth
    $w\geq n$, a
    shrink-factor $f\leq 1$, and a value $k>2$
    indicating that perturbations are performed using $k$-swaps.}
  \label{algorithm:swan}
\end{figure}

\clearpage

\begin{landscape}
  \begin{table}
  \center
    \caption{\linespread{0.9}\selectfont Running Gurobi with the encoding from
      Section~\ref{section:csp} on all instances from the benchmarkset
      described in Section~\ref{sec:benchmarkset}. An entry for the
      lower bound (LB), upper bound (UB), or the gap between the two is
      bold (and \textcolor{fg}{green} for the gap) if it is the best found within the table. The last three
      columns contain these virtual best values. A column with the
      suffix ``-pp'' means that the preprocessing of
      Section~\ref{section:kernelization} is applied before the solver
    is invoked; ``-init'' means that Gurobi is initialized with a
    solution found by \Algo{b-swan}; and ``-swan'' indicates that we
    improve every intermediate solution reported by Gurobi using
    \Algo{swan}. Runs marked with ``TO'' were terminated after 4
    hours, runs marked with ``MO'' used more than 10 GB of memory.}
    \label{table:gurobi}
    \scalebox{0.55}{\begin{tabular}{@{}ll|lllr|lllr|lllr|lllr|lllr|lll@{}}
\toprule
   &     & Gurobi             &                       &                                      & \multicolumn{1}{l|}{}         & Gurobi-pp          &                       &                                      & \multicolumn{1}{l|}{}         & Gurobi-init         &                     &                                       & \multicolumn{1}{l|}{}         & \multicolumn{2}{l}{Gurobi-swan}          &                                      & \multicolumn{1}{l|}{}         & \multicolumn{2}{l}{Gurobi-pp-swan}       &                                       & \multicolumn{1}{l|}{}         &            &            &        \\
$|\Alpha|$  & $|\tgrid|$   & LB                 & UB                    & Gap                                  & \multicolumn{1}{l|}{Time (s)} & LB                 & UB                    & Gap                                  & \multicolumn{1}{l|}{Time (s)} & LB                  & UB                  & Gap                                   & \multicolumn{1}{l|}{Time (s)} & LB                 & UB                  & Gap                                  & \multicolumn{1}{l|}{Time (s)} & LB                 & UB                  & Gap                                   & \multicolumn{1}{l|}{Time (s)} & VLB      & VUB      & VGap \\ \midrule
\multicolumn{25}{c}{Asteroid-belt instances} \\  \midrule
5  & 6   & \textbf{12318.13}  & \textbf{12318.13}     & {\color{fg} \textbf{0.00}} & 0.3                           & \textbf{12318.13}  & \textbf{12318.13}     & {\color{fg} \textbf{0.00}} & 0.04                          & \textbf{12318.13}   & \textbf{12318.13}   & {\color{fg} \textbf{0.00}}  & 0.03                          & \textbf{12318.13}  & \textbf{12318.13}   & {\color{fg} \textbf{0.00}} & 0.04                          & \textbf{12318.13}  & \textbf{12318.13}   & {\color{fg} \textbf{0.00}}  & 0.03                          & 12318.13   & 12318.13   & 0.00   \\
5  & 11  & \textbf{10938.92}  & \textbf{10938.92}     & {\color{fg} \textbf{0.00}} & 0.07                          & \textbf{10938.92}  & \textbf{10938.92}     & {\color{fg} \textbf{0.00}} & 0.05                          & \textbf{10938.92}   & \textbf{10938.92}   & {\color{fg} \textbf{0.00}}  & 0.05                          & \textbf{10938.92}  & \textbf{10938.92}   & {\color{fg} \textbf{0.00}} & 0.1                           & \textbf{10938.92}  & \textbf{10938.92}   & {\color{fg} \textbf{0.00}}  & 0.04                          & 10938.92   & 10938.92   & 0.00   \\
5  & 21  & \textbf{10533.94}  & \textbf{10533.94}     & {\color{fg} \textbf{0.00}} & 0.1                           & \textbf{10533.94}  & \textbf{10533.94}     & {\color{fg} \textbf{0.00}} & 0.08                          & \textbf{10533.94}   & \textbf{10533.94}   & {\color{fg} \textbf{0.00}}  & 0.1                           & \textbf{10533.94}  & \textbf{10533.94}   & {\color{fg} \textbf{0.00}} & 0.1                           & \textbf{10533.94}  & \textbf{10533.94}   & {\color{fg} \textbf{0.00}}  & 0.05                          & 10533.94   & 10533.94   & 0.00   \\
5  & 41  & \textbf{10459.16}  & \textbf{10459.16}     & {\color{fg} \textbf{0.00}} & 0.2                           & \textbf{10459.16}  & \textbf{10459.16}     & {\color{fg} \textbf{0.00}} & 0.1                           & \textbf{10459.16}   & \textbf{10459.16}   & {\color{fg} \textbf{0.00}}  & 0.3                           & \textbf{10459.16}  & \textbf{10459.16}   & {\color{fg} \textbf{0.00}} & 0.4                           & \textbf{10459.16}  & \textbf{10459.16}   & {\color{fg} \textbf{0.00}}  & 0.08                          & 10459.16   & 10459.16   & 0.00   \\
5  & 81  & \textbf{10431.62}  & \textbf{10431.62}     & {\color{fg} \textbf{0.00}} & 0.8                           & \textbf{10431.62}  & \textbf{10431.62}     & {\color{fg} \textbf{0.00}} & 0.4                           & \textbf{10431.62}   & \textbf{10431.62}   & {\color{fg} \textbf{0.00}}  & 1.2                           & \textbf{10431.62}  & \textbf{10431.62}   & {\color{fg} \textbf{0.00}} & 1.3                           & \textbf{10431.62}  & \textbf{10431.62}   & {\color{fg} \textbf{0.00}}  & 0.2                           & 10431.62   & 10431.62   & 0.00   \\
5  & 161 & \textbf{10431.60}  & \textbf{10431.60}     & {\color{fg} \textbf{0.00}} & 3.4                           & \textbf{10431.60}  & \textbf{10431.60}     & {\color{fg} \textbf{0.00}} & 1.9                           & \textbf{10431.60}   & \textbf{10431.60}   & {\color{fg} \textbf{0.00}}  & 5.7                           & \textbf{10431.60}  & \textbf{10431.60}   & {\color{fg} \textbf{0.00}} & 5.8                           & \textbf{10431.60}  & \textbf{10431.60}   & {\color{fg} \textbf{0.00}}  & 0.8                           & 10431.60   & 10431.60   & 0.00   \\ \midrule
10 & 11  & \textbf{38734.94}  & \textbf{38734.94}     & {\color{fg} \textbf{0.00}} & 0.2                           & \textbf{38734.94}  & \textbf{38734.94}     & {\color{fg} \textbf{0.00}} & 0.08                          & \textbf{38734.94}   & \textbf{38734.94}   & {\color{fg} \textbf{0.00}}  & 0.1                           & \textbf{38734.94}  & \textbf{38734.94}   & {\color{fg} \textbf{0.00}} & 0.3                           & \textbf{38734.94}  & \textbf{38734.94}   & {\color{fg} \textbf{0.00}}  & 0.07                          & 38734.94   & 38734.94   & 0.00   \\
10 & 21  & \textbf{28173.63}  & \textbf{28173.63}     & {\color{fg} \textbf{0.00}} & 0.5                           & \textbf{28173.63}  & \textbf{28173.63}     & {\color{fg} \textbf{0.00}} & 0.2                           & \textbf{28173.63}   & \textbf{28173.63}   & {\color{fg} \textbf{0.00}}  & 0.4                           & \textbf{28173.63}  & \textbf{28173.63}   & {\color{fg} \textbf{0.00}} & 0.9                           & \textbf{28173.63}  & \textbf{28173.63}   & {\color{fg} \textbf{0.00}}  & 0.2                           & 28173.63   & 28173.63   & 0.00   \\
10 & 41  & \textbf{25290.47}  & \textbf{25290.47}     & {\color{fg} \textbf{0.00}} & 1.6                           & \textbf{25290.47}  & \textbf{25290.47}     & {\color{fg} \textbf{0.00}} & 1.2                           & \textbf{25290.47}   & \textbf{25290.47}   & {\color{fg} \textbf{0.00}}  & 1.7                           & \textbf{25290.47}  & \textbf{25290.47}   & {\color{fg} \textbf{0.00}} & 2.1                           & \textbf{25290.47}  & \textbf{25290.47}   & {\color{fg} \textbf{0.00}}  & 0.6                           & 25290.47   & 25290.47   & 0.00   \\
10 & 81  & \textbf{24850.79}  & \textbf{24850.79}     & {\color{fg} \textbf{0.00}} & 11.6                          & \textbf{24850.79}  & \textbf{24850.79}     & {\color{fg} \textbf{0.00}} & 5.7                           & \textbf{24850.79}   & \textbf{24850.79}   & {\color{fg} \textbf{0.00}}  & 8.2                           & \textbf{24850.79}  & \textbf{24850.79}   & {\color{fg} \textbf{0.00}} & 8.4                           & \textbf{24850.79}  & \textbf{24850.79}   & {\color{fg} \textbf{0.00}}  & 5.1                           & 24850.79   & 24850.79   & 0.00   \\
10 & 161 & \textbf{24718.40}  & \textbf{24718.40}     & {\color{fg} \textbf{0.00}} & 55.1                          & \textbf{24718.40}  & \textbf{24718.40}     & {\color{fg} \textbf{0.00}} & 48.1                          & \textbf{24718.40}   & \textbf{24718.40}   & {\color{fg} \textbf{0.00}}  & 68.7                          & \textbf{24718.40}  & \textbf{24718.40}   & {\color{fg} \textbf{0.00}} & 57.5                          & \textbf{24718.40}  & \textbf{24718.40}   & {\color{fg} \textbf{0.00}}  & 39                            & 24718.40   & 24718.40   & 0.00   \\ \midrule
20 & 21  & \textbf{118104.77} & \textbf{118104.77}    & {\color{fg} \textbf{0.00}} & 5.1                           & \textbf{118104.77} & \textbf{118104.77}    & {\color{fg} \textbf{0.00}} & 5.2                           & \textbf{118104.77}  & \textbf{118104.77}  & {\color{fg} \textbf{0.00}}  & 3.9                           & \textbf{118104.77} & \textbf{118104.77}  & {\color{fg} \textbf{0.00}} & 5.1                           & \textbf{118104.77} & \textbf{118104.77}  & {\color{fg} \textbf{0.00}}  & 4.9                           & 118104.77  & 118104.77  & 0.00   \\
20 & 41  & \textbf{95033.47}  & \textbf{95033.47}     & {\color{fg} \textbf{0.00}} & 510                           & \textbf{95033.47}  & \textbf{95033.47}     & {\color{fg} \textbf{0.00}} & 574                           & \textbf{95033.47}   & \textbf{95033.47}   & {\color{fg} \textbf{0.00}}  & 904                           & \textbf{95033.47}  & \textbf{95033.47}   & {\color{fg} \textbf{0.00}} & 654                           & \textbf{95033.47}  & \textbf{95033.47}   & {\color{fg} \textbf{0.00}}  & 634                           & 95033.47   & 95033.47   & 0.00   \\
20 & 81  & 79334.00           & \textbf{92544.51}     & 14.27                                & TO                            & 80382.53           & \textbf{92544.51}     & 13.14                                & TO                            & \textbf{80955.44}   & \textbf{92544.51}   & {\color{fg} \textbf{12.52}} & TO                            & 79402.51           & \textbf{92544.51}   & 14.20                                & TO                            & 80308.53           & \textbf{92544.51}   & 13.22                                 & TO                            & 80955.44   & 92544.51   & 12.52  \\
20 & 161 & 70600.96           & 94066.26              & 24.95                                & TO                            & 70548.66           & \textbf{91631.03}     & 23.01                                & TO                            & 70627.11            & \textbf{91631.03}   & 22.92                                 & TO                            & 70370.22           & \textbf{91631.03}   & 23.20                                & TO                            & \textbf{70628.46}  & \textbf{91631.03}   & {\color{fg} \textbf{22.92}} & TO                            & 70628.46   & 91631.03   & 22.92  \\ \midrule
40 & 41  & \textbf{391076.40} & \textbf{391076.40}    & {\color{fg} \textbf{0.00}} & 1157                          & \textbf{391076.40} & \textbf{391076.40}    & {\color{fg} \textbf{0.00}} & 1452                          & \textbf{391076.40}  & \textbf{391076.40}  & {\color{fg} \textbf{0.00}}  & 1151                          & \textbf{391076.40} & \textbf{391076.40}  & {\color{fg} \textbf{0.00}} & 1183                          & \textbf{391076.40} & \textbf{391076.40}  & {\color{fg} \textbf{0.00}}  & 1227                          & 391076.40  & 391076.40  & 0.00   \\
40 & 81  & 242883.61          & \multicolumn{1}{c}{-} & \multicolumn{1}{c}{-}                & TO                            & 253560.69          & 337169.25             & 24.80                                & TO                            & 254675.76           & 318094.39           & 19.94                                 & TO                            & \textbf{259178.21} & 317613.35           & 18.40                                & TO                            & 256115.81          & \textbf{316923.28}  & 19.19                                 & TO                            & 259178.21  & 316923.28  & 18.22  \\
40 & 161 & 220675.01          & \multicolumn{1}{c}{-} & \multicolumn{1}{c}{-}                & MO                            & 224163.88          & 382472.46             & 41.39                                & TO                            & 223818.05           & 316262.24           & 29.23                                 & MO                            & 224098.40          & 313045.86           & 28.41                                & TO                            & \textbf{224816.16} & \textbf{312500.91}  & {\color{fg} \textbf{28.06}} & TO                            & 224816.16  & 312500.91  & 28.06  \\ \midrule
80 & 81  & 1159711.10         & \multicolumn{1}{c}{-} & \multicolumn{1}{c}{-}                & TO                            & 1159455.64         & \multicolumn{1}{c}{-} & \multicolumn{1}{c}{-}                & TO                            & \textbf{1236138.36} & 1557734.36          & 20.65                                 & TO                            & 1235436.20         & \textbf{1557280.92} & 20.67                                & TO                            & 1218417.46         & 1559213.76          & 21.86                                 & TO                            & 1236138.36 & 1557280.92 & 20.62  \\
80 & 161 & 0.00               & \multicolumn{1}{c}{-} & \multicolumn{1}{c}{-}                & MO                            & \textbf{762357.31} & \multicolumn{1}{c}{-} & \multicolumn{1}{c}{-}                & MO                            & 0.00                & \textbf{1420767.83} & 100.00                                & MO                            & 0.00               & 1420767.83          & 100.00                               & MO                            & \textbf{762357.31} & \textbf{1420767.83} & {\color{fg} \textbf{46.34}} & MO                            & 762357.31  & 1420767.83 & 46.34  \\ \midrule
\multicolumn{25}{c}{Jovian-moon instances} \\  \midrule
4  & 6   & \textbf{38705.57}  & \textbf{38705.57}     & {\color{fg} \textbf{0.00}} & 0.05                          & \textbf{38705.57}  & \textbf{38705.57}     & {\color{fg} \textbf{0.00}} & 0.04                          & \textbf{38705.57}   & \textbf{38705.57}   & {\color{fg} \textbf{0.00}}  & 0.03                          & \textbf{38705.57}  & \textbf{38705.57}   & {\color{fg} \textbf{0.00}} & 0.05                          & \textbf{38705.57}  & \textbf{38705.57}   & {\color{fg} \textbf{0.00}}  & 0.03                          & 38705.57   & 38705.57   & 0.00   \\
4  & 11  & \textbf{20096.59}  & \textbf{20096.59}     & {\color{fg} \textbf{0.00}} & 0.06                          & \textbf{20096.59}  & \textbf{20096.59}     & {\color{fg} \textbf{0.00}} & 0.04                          & \textbf{20096.59}   & \textbf{20096.59}   & {\color{fg} \textbf{0.00}}  & 0.04                          & \textbf{20096.59}  & \textbf{20096.59}   & {\color{fg} \textbf{0.00}} & 0.1                           & \textbf{20096.59}  & \textbf{20096.59}   & {\color{fg} \textbf{0.00}}  & 0.09                          & 20096.59   & 20096.59   & 0.00   \\
4  & 21  & \textbf{18618.00}  & \textbf{18618.00}     & {\color{fg} \textbf{0.00}} & 0.1                           & \textbf{18618.00}  & \textbf{18618.00}     & {\color{fg} \textbf{0.00}} & 0.05                          & \textbf{18618.00}   & \textbf{18618.00}   & {\color{fg} \textbf{0.00}}  & 0.1                           & \textbf{18618.00}  & \textbf{18618.00}   & {\color{fg} \textbf{0.00}} & 0.7                           & \textbf{18618.00}  & \textbf{18618.00}   & {\color{fg} \textbf{0.00}}  & 0.9                           & 18618.00   & 18618.00   & 0.00   \\
4  & 41  & \textbf{17841.49}  & \textbf{17841.49}     & {\color{fg} \textbf{0.00}} & 0.2                           & \textbf{17841.49}  & \textbf{17841.49}     & {\color{fg} \textbf{0.00}} & 0.08                          & \textbf{17841.49}   & \textbf{17841.49}   & {\color{fg} \textbf{0.00}}  & 0.2                           & \textbf{17841.49}  & \textbf{17841.49}   & {\color{fg} \textbf{0.00}} & 2.3                           & \textbf{17841.49}  & \textbf{17841.49}   & {\color{fg} \textbf{0.00}}  & 2.1                           & 17841.49   & 17841.49   & 0.00   \\
4  & 81  & \textbf{17623.63}  & \textbf{17623.63}     & {\color{fg} \textbf{0.00}} & 0.9                           & \textbf{17623.63}  & \textbf{17623.63}     & {\color{fg} \textbf{0.00}} & 0.1                           & \textbf{17623.63}   & \textbf{17623.63}   & {\color{fg} \textbf{0.00}}  & 1                             & \textbf{17623.63}  & \textbf{17623.63}   & {\color{fg} \textbf{0.00}} & 9                             & \textbf{17623.63}  & \textbf{17623.63}   & {\color{fg} \textbf{0.00}}  & 5.7                           & 17623.63   & 17623.63   & 0.00   \\
4  & 161 & \textbf{17430.97}  & \textbf{17430.97}     & {\color{fg} \textbf{0.00}} & 3.6                           & \textbf{17430.97}  & \textbf{17430.97}     & {\color{fg} \textbf{0.00}} & 0.3                           & \textbf{17430.97}   & \textbf{17430.97}   & {\color{fg} \textbf{0.00}}  & 3.6                           & \textbf{17430.97}  & \textbf{17430.97}   & {\color{fg} \textbf{0.00}} & 18.61                         & \textbf{17430.97}  & \textbf{17430.97}   & {\color{fg} \textbf{0.00}}  & 20.4                          & 17430.97   & 17430.97   & 0.00   \\ \midrule
10 & 11  & \textbf{19871.59}  & \textbf{19871.59}     & {\color{fg} \textbf{0.00}} & 0.2                           & \textbf{19871.59}  & \textbf{19871.59}     & {\color{fg} \textbf{0.00}} & 0.09                          & \textbf{19871.59}   & \textbf{19871.59}   & {\color{fg} \textbf{0.00}}  & 0.2                           & \textbf{19871.59}  & \textbf{19871.59}   & {\color{fg} \textbf{0.00}} & 0.4                           & \textbf{19871.59}  & \textbf{19871.59}   & {\color{fg} \textbf{0.00}}  & 0.2                           & 19871.59   & 19871.59   & 0.00   \\
10 & 21  & \textbf{16268.25}  & \textbf{16268.25}     & {\color{fg} \textbf{0.00}} & 0.8                           & \textbf{16268.25}  & \textbf{16268.25}     & {\color{fg} \textbf{0.00}} & 0.4                           & \textbf{16268.25}   & \textbf{16268.25}   & {\color{fg} \textbf{0.00}}  & 0.5                           & \textbf{16268.25}  & \textbf{16268.25}   & {\color{fg} \textbf{0.00}} & 1.1                           & \textbf{16268.25}  & \textbf{16268.25}   & {\color{fg} \textbf{0.00}}  & 0.5                           & 16268.25   & 16268.25   & 0.00   \\
10 & 41  & \textbf{15611.09}  & \textbf{15611.09}     & {\color{fg} \textbf{0.00}} & 2.5                           & \textbf{15611.09}  & \textbf{15611.09}     & {\color{fg} \textbf{0.00}} & 1.3                           & \textbf{15611.09}   & \textbf{15611.09}   & {\color{fg} \textbf{0.00}}  & 2                             & \textbf{15611.09}  & \textbf{15611.09}   & {\color{fg} \textbf{0.00}} & 2.4                           & \textbf{15611.09}  & \textbf{15611.09}   & {\color{fg} \textbf{0.00}}  & 0.9                           & 15611.09   & 15611.09   & 0.00   \\
10 & 81  & \textbf{15443.86}  & \textbf{15443.86}     & {\color{fg} \textbf{0.00}} & 12.7                          & \textbf{15443.86}  & \textbf{15443.86}     & {\color{fg} \textbf{0.00}} & 7.3                           & \textbf{15443.86}   & \textbf{15443.86}   & {\color{fg} \textbf{0.00}}  & 12.3                          & \textbf{15443.86}  & \textbf{15443.86}   & {\color{fg} \textbf{0.00}} & 14.4                          & \textbf{15443.86}  & \textbf{15443.86}   & {\color{fg} \textbf{0.00}}  & 7.3                           & 15443.86   & 15443.86   & 0.00   \\
10 & 161 & \textbf{15373.88}  & \textbf{15373.88}     & {\color{fg} \textbf{0.00}} & 92                            & \textbf{15373.88}  & \textbf{15373.88}     & {\color{fg} \textbf{0.00}} & 37.2                          & \textbf{15373.88}   & \textbf{15373.88}   & {\color{fg} \textbf{0.00}}  & 128                           & \textbf{15373.88}  & \textbf{15373.88}   & {\color{fg} \textbf{0.00}} & 155                           & \textbf{15373.88}  & \textbf{15373.88}   & {\color{fg} \textbf{0.00}}  & 51.8                          & 15373.88   & 15373.88   & 0.00   \\ \midrule
20 & 21  & \textbf{50283.81}  & \textbf{50283.81}     & {\color{fg} \textbf{0.00}} & 6.4                           & \textbf{50283.81}  & \textbf{50283.81}     & {\color{fg} \textbf{0.00}} & 4.5                           & \textbf{50283.81}   & \textbf{50283.81}   & {\color{fg} \textbf{0.00}}  & 3.9                           & \textbf{50283.81}  & \textbf{50283.81}   & {\color{fg} \textbf{0.00}} & 5.6                           & \textbf{50283.81}  & \textbf{50283.81}   & {\color{fg} \textbf{0.00}}  & 3.9                           & 50283.81   & 50283.81   & 0.00   \\
20 & 41  & \textbf{46505.23}  & \textbf{46505.23}     & {\color{fg} \textbf{0.00}} & 692                           & \textbf{46505.23}  & \textbf{46505.23}     & {\color{fg} \textbf{0.00}} & 306                           & \textbf{46505.23}   & \textbf{46505.23}   & {\color{fg} \textbf{0.00}}  & 322                           & \textbf{46505.23}  & \textbf{46505.23}   & {\color{fg} \textbf{0.00}} & 340                           & \textbf{46505.23}  & \textbf{46505.23}   & {\color{fg} \textbf{0.00}}  & 268                           & 46505.23   & 46505.23   & 0.00   \\
20 & 81  & 42003.30           & \textbf{44797.84}     & 6.24                                 & TO                            & 43242.94           & \textbf{44797.84}     & 3.47                                 & TO                            & 42963.33            & \textbf{44797.84}   & 4.10                                  & TO                            & 42406.85           & 44925.25            & 5.61                                 & TO                            & \textbf{43587.94}  & \textbf{44797.84}   & {\color{fg} \textbf{2.70}}  & TO                            & 43587.94   & 44797.84   & 2.70   \\
20 & 161 & 37554.80           & 44602.57              & 15.80                                & TO                            & 37748.77           & \textbf{44340.50}     & 14.87                                & TO                            & 37501.51            & 44602.57            & 15.92                                 & TO                            & 37504.14           & 44602.57            & 15.91                                & TO                            & \textbf{38028.42}  & 44602.57            & 14.74                                 & TO                            & 38028.42   & 44340.50   & 14.24  \\ \bottomrule
\end{tabular}
}
  \end{table}
\end{landscape}

\end{document}